\documentclass[12pt,a4paper]{amsart}
\makeatletter
\renewcommand\normalsize{%
    \@setfontsize\normalsize{11.7}{14pt plus .3pt minus .3pt}%
    \abovedisplayskip 10\p@ \@plus4\p@ \@minus4\p@
    \abovedisplayshortskip 6\p@ \@plus2\p@
    \belowdisplayshortskip 6\p@ \@plus2\p@
    \belowdisplayskip \abovedisplayskip}
\renewcommand\small{%
    \@setfontsize\small{9.5}{12\p@ plus .2\p@ minus .2\p@}%
    \abovedisplayskip 8.5\p@ \@plus4\p@ \@minus1\p@
    \belowdisplayskip \abovedisplayskip
    \abovedisplayshortskip \abovedisplayskip
    \belowdisplayshortskip \abovedisplayskip}
\renewcommand\footnotesize{%
    \@setfontsize\footnotesize{8.5}{9.25\p@ plus .1pt minus .1pt}
    \abovedisplayskip 6\p@ \@plus4\p@ \@minus1\p@
    \belowdisplayskip \abovedisplayskip
    \abovedisplayshortskip \abovedisplayskip
    \belowdisplayshortskip \abovedisplayskip}
\ifdefined\pdfpagewidth
\else
\fi
\calclayout
\makeatother

\RequirePackage[nonfrench]{kotex}
\usepackage{kotex}

\usepackage{epsfig}
\usepackage{amsfonts}
\usepackage{amssymb}
\usepackage{amsmath}
\usepackage{amsthm}
\usepackage{amscd}
\usepackage{amstext}
\usepackage{latexsym}
\usepackage{enumerate}

\usepackage{alphalph}

\usepackage{array}
\usepackage{graphicx, subcaption}
\usepackage{todonotes}

\usepackage{xy}

\input xy
\xyoption{all}

 \usepackage[backref=false,bookmarks=false]{hyperref}

\newcommand{\ZZ}{\mathbb{Z}}
\newcommand{\CC}{\mathbb{C}}
\newcommand{\QQ}{\mathbb{Q}}
\newcommand{\RR}{\mathbb{R}}

\newcommand{\PP}{\mathbf{P}}
\newcommand{\JJ}{\mathcal{J}}

\newcommand{\OO}{\mathcal{O}}

\newcommand{\iddb}{\sqrt{-1} \partial \overline{\partial}}

\newcommand{\db}{\overline{\partial}}
\newcommand{\al}{\alpha}

\newcommand{\ld}{\lambda}

\newcommand{\qa}{\quad}
\newcommand{\vp}{\varphi}

\newcommand{\yr}{ Y_{\reg}}

\newcommand{\Hom}{\mathcal Hom}

\newcommand{\noi}{\noindent}

\providecommand{\wt}[1]{\widetilde{#1}}

\providecommand{\abs}[1]{\left|#1\right|}

\providecommand{\norm}[1]{\lVert#1\rVert}

\theoremstyle{plain}

\newtheorem{theorem}{Theorem}[section]

\newtheorem{lemma}[theorem]{Lemma}

\newtheorem{proposition}[theorem]{Proposition}
\newtheorem{definition}[theorem]{Definition}

\newtheorem{question}[theorem]{Question}

 \newtheorem{example}[theorem]{\textnormal{\textbf{Example}}}

\theoremstyle{remark}
\newtheorem{remark1}[theorem]{Remark}

\DeclareMathOperator{\red}{red}
\DeclareMathOperator{\Supp}{Supp}
\DeclareMathOperator{\divisor}{div}

\DeclareMathOperator{\reg}{reg}

\DeclareMathOperator{\ord}{ord}

\DeclareMathOperator{\loc}{loc}
\DeclareMathOperator{\rank}{rank}

\begin{document}

\title[Canonical bundle formula, degenerating families of volume forms]{Canonical bundle formula and \\ degenerating families of volume forms}



\author{Dano Kim}

\date{}

\maketitle

\begin{abstract}

\noi Canonical bundle formula due to Kawamata and others has played fundamental roles in algebraic geometry. We show that the  canonical bundle formula has analytic characterization in terms of fiberwise integration, which confirms a folklore conjecture. The proof uses $L^2$ metrics and the valuative equivalence of plurisubharmonic singularities. As an application, we identify the singularity of the Ohsawa measure in a general $L^2$ extension theorem of Demailly for log canonical pairs. As another consequence, we give a partial answer to a question of Berndtsson on semipositivity theorems.

\end{abstract}


\setcounter{tocdepth}{2}


\section{Introduction}

Canonical bundle formula~\cite{Ka98} (cf. \cite{FM00}, \cite{Am04}, \cite{Ko07}) has played fundamental roles in higher dimensional algebraic geometry. The prototype of such formula arises from the classical work of Kodaira on elliptic fibrations \cite[Thm. 12]{K64}. The setting of the canonical bundle formula and our main results is as follows.

  Let $f: X \to Y$ be a  surjective projective morphism with connected fibers between complex manifolds. Simple normal crossing divisors $R$ on $X$ and $B$ on $Y$ satisfy Kawamata's condition in Definition~\ref{snc} so that  $f$ is smooth over $Y \setminus B$. 
    Kawamata's canonical bundle formula is the following equality 
  
  \begin{equation*}
  K_X + R \sim_{\QQ} f^* ( K_Y + B_R +  J(X \slash Y, R)). 
\end{equation*} 
  where $B_R$ is the discriminant divisor associated to $R$ and $J(X \slash Y, R)$ is the moduli part line bundle.

The discriminant divisor $B_R$ is a particular linear combination  of the components of $B$, which reflects the singularities of both $R$ and the singular fibers of $f$, generalizing the case of elliptic fibrations.  The definition of $B_R$ is motivated within algebraic geometry :  according to \cite{Ka19},  the coefficients of the discriminant divisor  were defined so that they behave well under semi-stable reduction.

  In this paper, we show that $B_R$ has analytic characterization in terms of fiberwise integration along $f$ in Theorem~\ref{fi}. This is obtained via related results on $L^2$ metrics, which are metrics determined by fiberwise integration, such as the following Theorem~\ref{main}.

 \subsection{Statements of the main results} \hfill \\

  In the above setting, let $L := \OO(R)$ be the $\QQ$-line bundle associated to $R$. Let $M$ be the $\QQ$-line bundle on $Y$ defined by the relation $K_X + L = f^* (K_Y + M)$. \footnote{See \textbf{conventions} in the beginning of Section 2 for our notation of hermitian line bundles.}    Our first main result is the following metric version of the canonical bundle formula~\cite{Ka98}. 
   

 \begin{theorem}\label{main}
	
	 Let $f: (X,R) \to (Y,B)$ be a  surjective projective morphism with connected fibers between complex manifolds satisfying Kawamata's condition, Definition~\ref{snc}. 
	 
	 \begin{itemize}
	 \item Assume that the horizontal divisor $R_h$ is effective, so that its coefficients are  in the interval $[0,1)$. 
	
	\item Let  $\ld$  be a singular hermitian metric of $L$ given by (a defining meromorphic section of) the divisor $R$.
	Let $\mu$ be the $L^2$ metric (induced from $\ld$) for the $\QQ$-line bundle $M$ in Definition~\ref{L2metricQ}. 
	\end{itemize}
	 
\noi	 Then $\mu$ is equal to the product of two singular hermitian metrics $(\OO(B_R), \eta)$ and $(J := M - \OO(B_R), \psi)$, i.e. 
	  
	 \begin{equation*}\label{asy}
	  \mu = \eta + \psi \; \; \; \; ( e^{-\mu} = e^{-\eta} e^{-\psi})
	  \end{equation*}
	   where
	   \begin{itemize}
	   \item
	    $\eta$ is a singular hermitian metric given by the discriminant divisor $B_R$ and
	    \item $\psi$ is a singular hermitian metric of $J$ with semipositive curvature current and with  zero Lelong number at every point of $Y$. 
 \end{itemize}
\end{theorem}

\noi In particular, if $Y$ is compact, then the moduli part line bundle $J$ is nef by Proposition~\ref{psa}. This recovers nefness of $J$ in \cite{Ka98} with a different method in this case.  Theorem~\ref{main} answers a question of \cite[p.742]{EFM} which asked if there is a metric approach to the canonical bundle formula~\cite{Ka98}. 

 Note that $R$ and $B_R$ are not necessarily effective, thus $\ld, \mu$ and $\eta$ are not necessarily with semipositive curvature currents. Also while $\mu$ is uniquely determined by $\ld$, the metrics $\psi$ and $\eta$ are not: they  can be added with constants $c$ and $-c$, for example.

\noi \emph{Remark.} In order to make the comparison with this paper explicitly clear,  we emphasize that  \cite{Ka98} did not use analytic notions such as singular hermitian metrics or plurisubharmonic functions at all. 
 \\

Our next main result shows that the discriminant divisor has analytic characterization in terms of fiberwise integration along $f$ of singular volume forms. This confirms a folklore conjecture, cf. \cite[p.81]{Ka96}, \cite{Ka98}, \cite[p.10]{Ka00A}, \cite{T07}.

\begin{theorem}\label{fi}
 Let $f: (X,R) \to (Y,B)$ be a  surjective projective morphism with connected fibers between complex manifolds satisfying Kawamata's condition, Definition~\ref{snc}.
  Let $B_R$ be the discriminant divisor of $R$. Let $m = \dim Y$.
 \begin{itemize}
 
\item
    Let $\alpha$ and $\beta$ be  singular volume forms on $X$ with divisorial poles along  $R = \sum a_i R_i$ and along $R + \lceil -R_h \rceil$ respectively,  \emph{(cf. \eqref{volform}, \S 2.3)}.
    
    \item   Let $f_* \alpha$ and $f_* \beta$ be the singular volume forms on $Y_0$ obtained as the fiberwise integration of $\alpha$ and $\beta$ respectively along $f: X_0 := f^{-1} (Y_0) \to Y_0$ where $Y_0 \subset Y$ is the subset of regular values of $f$.
  
  \end{itemize}
  
\noi  Then the following hold. 
    
    \begin{enumerate}
    
    \item
    
     The fiberwise integration $f_* \beta$   has equivalent singularities (on $Y_0$) to a singular volume form $\theta$ (defined on $Y$) with divisorial poles along  $B_R$ up to a  plurisubharmonic weight with vanishing Lelong numbers.  In other words, 
 $\theta$ is locally written as 
	  
	  \begin{equation*}\label{vvv}
	   \theta(w) = g(w) ( \prod^m_{i=1}  \abs{w_i}^{-2 a_i} )  e^{-\psi(w)} \abs{dw_1 \wedge \ldots \wedge dw_m}^2   
	  \end{equation*}
 where local analytic coordinates $w = (w_1, \ldots, w_m)$ are adapted to $B$ on an open relatively compact subset  $U \subset Y$,	$\psi$ is a plurisubharmonic function on $U$ with zero Lelong number at every point, $\sum a_i \divisor (w_i) = B_R $ on $U$ and $g: U \to \RR$ is a positive and bounded function. 
	  
	  \item
	  
	  The fiberwise integration $f_* \al$ is less singular (on $Y_0$) than $\theta$.
	  
	  \end{enumerate}
	  	
\end{theorem}

 The statement (1) includes the fiberwise integration version of Theorem~\ref{main} when $R_h \ge 0$ (hence $R+ \lceil -R_h \rceil = R$). The possible  difference of singularities in (2) between $f_* \al$ and $\theta$ is inevitable and results from components of the snc divisor $R_h$ that are not transversal with a special fiber of $f: X \to Y$ over a point of $B$, see Remark~\ref{mild}.  Also note that even if a component $(w_i = 0)$ of $B$ lies outside $Y_0$, the factor $\abs{w_i}^{-2 a_i}$ appears in \eqref{vvv}.

  \subsection{Outline of the proof of Theorems~\ref{main}, \ref{fi}}  \hfill \\

We use the $L^2$ metric defined in the setting of the theorems and its fiberwise integration property. 
 
 $\bullet$ \textbf{$L^2$ metrics.} In this paper, we use singular hermitian metrics only for $\QQ$-line bundles. 
 
 First, in the setting of Theorem~\ref{main}, we have the log Calabi-Yau condition $K_X + L = f^* (K_Y + M)$ from \eqref{Jpart}, \eqref{relation2} where $L = O(R)$. When $L$ and $M$ are $\ZZ$-line bundles, the direct image is  $M = f_* (K_{X/Y} + L)$ for which the $L^2$ metric is defined. The definition and the main properties of the $L^2$ metric are generalized when $L$ and $M$ are $\QQ$-line bundles in Section 3, cf. Proposition~\ref{well}. 
 
 $\bullet$ In the setting of Theorem~\ref{fi},   $R_h$ is not necessarily effective and we need to use the direct image $f_* (K_{X/Y} + L +A)$ where  $A = \OO(\lceil -R_h \rceil)$, in the case when $L$ and $M$ are $\ZZ$-line bundles. In the general $\QQ$-line bundle setting, the counterpart of the direct image  is the $\QQ$-line bundle $$M \otimes (f_* A)|_{Y_{1, \text{free}}}$$ for which our $L^2$ metric is defined.   Here the Zariski open subset $Y_{1, \text{free}} \subset Y$ is the locus  where $f_* A$ is locally free of rank $1$, see Theorem~\ref{pt}.

$\bullet$ \textbf{Valuative equivalence.} Once we have the $L^2$ metric, in Section 4, we proceed via identifying the singularities of the $L^2$ metric in terms of all divisorial valuations over $Y$. 
This is done in the key  result Theorem~\ref{valu} by carefully applying the defining condition of the discriminant divisor to show that, in fact, the $L^2$ metric is \emph{valuatively equivalent} to 
 the psh metric given by the discriminant divisor.

$\bullet$ \textbf{Siu decomposition.}  This is indeed a very effective viewpoint for our purpose: the psh metric $\psi$ with vanishing Lelong numbers in Theorem~\ref{main} is obtained  immediately from comparing the Siu decompositions of the curvature currents of the two metrics in comparison, cf. Lemma~\ref{decom}. Namely the divisor part of the Siu decomposition of the curvature current of the $L^2$ metric is precisely given by the discriminant divisor, hence the residual part $R_T$ itself gives rise to the psh metric $\psi$ for $J$ in the proof of Theorem~\ref{main} and Theorem~\ref{fi}. 

$\bullet$  We remark that in Theorem~\ref{fi}, it is $f_* \beta$ (instead of $f_* \al$) that has equivalent singularities with $\theta$. Note that the choice of adding $T:= \lceil -R_h \rceil$ to $R_h$ is uniquely determined (among $\QQ$-divisors supported on $R_h$) so that Theorem~\ref{valu} works: $T$ is the only $\ZZ$-divisor which makes $T+ R_h$ both effective and klt. The local integrability from the kltness is indeed used toward the end of the proof of Theorem~\ref{valu}.  
\\


\subsection{Applications to $L^2$ extension theorems}  \hfill \\

\noi 
 We now turn to the application of Theorem~\ref{fi} to $L^2$ extension theorems of Ohsawa-Takegoshi type in several complex variables.  
 
 Let $Y \subset X$ be a subvariety of a complex manifold. Let $L$ be a holomorphic line bundle on $X$ and $K_X$  the canonical line bundle of $X$.  An \emph{$L^2$ extension theorem} is a type of a statement that  (under suitable conditions on $X, Y, L, \ldots$)  if a certain $L^2$ norm $ \norm{s}_Y$   is finite for a holomorphic section $s$ on $Y$ of  $(K_X \otimes L)|_Y$, then there exists  $\tilde{s} \in H^0(X, K_X \otimes L)$   such that  $$\tilde{s}|_Y = s \; \text{ and } \; \norm{\tilde{s}}_X \le c \norm{s}_Y $$ for some constant $c>0$.
  Here the $L^2$ norm $ \norm{s}_Y$ plays the crucial role of determining whether $s$ can be extended or not : we will call it as the \textbf{input norm} of the particular statement of $L^2$ extension. It is  crucial to understand the criterion of extension given by the  finiteness of the norm  $\norm{s}_Y$. In general, such finiteness is nontrivial to achieve due to singularity of the measure against which the norm is taken.

 Since the seminal work of Ohsawa and Takegoshi~\cite{OT87}, there have been extensive developments on $L^2$ extension theorems   (cf. e.g.  \cite{OT87}, \cite{Ma93}, \cite{Si96}, \cite{D00},  \cite{O01},  \cite{K07} among many other works, cf. the introduction of \cite{K21} and references therein.)   Usually, $Y$ was taken to be the zero set of a section of a vecor bundle. This was extended to the case where $Y$ is (an irreducible component of) the subvariety of the multiplier ideal sheaf of a log canonical pair in \cite{K07}.

 In \cite[Theorem 2.8]{D15}, Demailly gave another $L^2$ extension theorem for a log canonical pair $(X, \Psi)$ where $\Psi$ is a quasi-psh function with analytic singularities (see Definition~\ref{Omeasure} for the detailed setting). Let $Y$ be an irreducible component of the non-klt locus of $(X, \Psi)$. The input norm  $\norm{s}_Y$ in \cite{D15}  is taken   with respect to the following {measure} defined in terms of a certain limit  taken over some neighborhoods shrinking to the pole set of $\Psi$.  
 
 \begin{definition}[=Definition~\ref{Omeasure}] \cite{O01},    \cite[Section 2]{D15} \label{Omeasure1}
Let $dV_X$ be a smooth volume form on $X$.   The  \textbf{ \emph{Ohsawa measure}} $dV[\Psi]_Y$ of $\Psi$ on $Y$ (with respect to $dV_X$) is 
a positive measure $d\mu$ on $\yr$ satisfying the following condition: 

 For every $g$, a real-valued compactly supported continuous function on $\yr$ and for every $\tilde{g}$, a compactly supported extension of $g$ to $X$, we have the relation 
 
 \begin{equation*}\label{dvp}
 \int_{\yr} g \; d\mu = \lim_{t \to -\infty} \int_{\{x \in X, \; t < \Psi(x) < t+1 \}} \tilde{g} e^{-\Psi} dV_{X}. 
\end{equation*}

\end{definition}

  The definition of $dV[\Psi]_Y$ is formulated precisely according to
 the well established (cf. \cite{OT87}, \cite{O01}, \cite{D15}) analytic methods of proof of the $L^2$ extension based on the $L^2$ estimates for $\db$ operator. Namely, once the input norm $\norm{s}_Y = \int_Y \abs{s} dV[\Psi]_Y$ is finite, then those analytic methods will apply to extend the given section $s$. However, this is at the price of leaving the singularity of $dV[\Psi]_Y$ yet to be understood in terms of geometry. Hence it is essential to consider the following natural question.

\begin{question}\label{whatis}

What is the singularity of the Ohsawa measure $dV[\Psi]_Y$? 

\end{question}

 As a first step, in \cite[Thm. 1.2]{K21},  some essential understanding of the Ohsawa measure was obtained: if $Y$ has non-unique log canonical places, then the Ohsawa measure is the `infinity measure' i.e. only the constant zero function will satisfy the input norm condition. This makes the $L^2$ extension result void on such $Y$, and leaves us the main case of a unique log canonical place.

 In this paper, we complete the geometric investigation of  the Ohsawa measure by applying Theorem~\ref{fi}  when $Y$ has a unique log canonical place. We show that the singularity (or its close upper bound) of the Ohsawa measure $dV[\Psi]_Y$ is  described   by the singularity appearing in Kawamata's subadjunction from algebraic geometry. 
 
\begin{theorem}[=see Theorem~\ref{dimage}] \label{dimage1}

  Let $(X, \Psi)$ and $Y$ be as above, with a unique log canonical place $E \to Y' \to Y$ as in \eqref{EYX}.   The Ohsawa measure $dV[\Psi]_Y$ is less singular than the direct image $\pi_* v$ under $\pi: Y' \to Y$ of a singular volume form $v$ on $Y'$ with divisorial poles along Kawamata's discriminant divisor $B_R$ on $Y'$ multiplied by a local psh weight with vanishing Lelong numbers.
  \end{theorem} 

 This answers Question~\ref{whatis} and confirms Demailly's expectation that ``the Ohsawa measure is an analytic concept but I believe it has potential meaning in algebraic geometry", \cite[46:40]{D16}. 
 
 Note that since $\pi$ is modification, the direct image $\pi_* v$ is merely identification with $v$ via the identity $\pi : Y'_0 \to Y_0$ for a Zariski open set $Y_0 \subset Y$. (The restrictions to $Y_0, Y'_0$ determine these measures, cf. the proof.)
 
Also Theorem~\ref{dimage1} says that $L^2$ extension theorems can indeed extend as many sections as would be expected from algebraic geometry in that, for example when $Y$ does not contain proper log canonical centers in it, $dV[\Psi]_Y$ presents no local obstruction for extension as $\pi_* v$ does so. 
\\

\noi \emph{Remark.}  One may ask whether the Ohsawa measure can be replaced with a less singular volume form in an (improved) $L^2$ extension theorem. We expect that Theorem~\ref{dimage1} can be used to give a negative answer in a similar manner as \cite[Thm. 1.7, 1.8]{KS23}.   
 
  On the other hand, a previous work of the author \cite{K07} first established an $L^2$ extension theorem for a log canonical pair, from a different approach not involving the Ohsawa measure. It used the measure $\pi_* v$ directly to define the input norm. 
  Theorem~\ref{dimage1} provides  a precise comparison of the input norms of two $L^2$ extension theorems : \cite[Thm. 4.2]{K07} and \cite[Thm. 2.8]{D15} both for a log canonical pair on a smooth projective variety, see Theorem~\ref{kawaresi}. This was our initial motivation for Theorem~\ref{fi}.

\subsection{Applications to semipositivity theorems} \hfill \\

\noi 
As another consequence of Theorem~\ref{main}, we give a partial answer to a question of Berndtsson (2005) who asked whether it would be possible to use the theory of positivity of direct images (along the line of \cite{B09}, \cite{BP08}, \cite{PT}, cf. Section 3) in the proof of the celebrated semipositivity theorems in  \cite{Ka81}, \cite{Ka00} and others, which used Hodge theory. 
By establishing Theorem~\ref{Kawa00}, we answer this question when the direct image of the relative canonical line bundle is locally free of rank $1$. 
\\

This paper is organized as follows. In Section 2, we have preliminaries on plurisubharmonic functions, singular hermitian metrics and singular volume forms. In Section 3, we define and show the existence of the $L^2$ metric in our setting. In Section 4, we prove Theorem~\ref{valu}, the key result of valuative equivalence between the $L^2$ metric and the discriminant divisor. In Section 5, we prove Theorem~\ref{main} and Theorem~\ref{fi} using Theorem~\ref{valu}. Sections 6 and 7 deal with the applications to $L^2$ extension theorems and to semipositivity theorems respectively. 
\\

\noi \emph{Remark.} Theorem~\ref{main} in particular recovers nefness of $J$ in \cite{Ka98} by Proposition~\ref{psa}. Such refinement of nefness was also given in a recent work \cite{T22} by a different method. \footnote{We mention that the current paper first appeared on arXiv and \cite{T22} was received by a journal, both in October 2019, whereas a manuscript of the current paper was brought to the attention of the author of \cite{T22} in March 2019. We note that the approach of the current paper is not `similar', as opposed to a mention in \cite{T22}. }

 \qa

\noi \textbf{Acknowledgment.} 
 The author would like to thank O. Fujino and J. Kollár for answering his questions and having helpful discussions. The author is also grateful to B. Berndtsson, J.-P. Demailly,   Y. Kawamata, C. Schnell,   V. Tosatti and  K.-I. Yoshikawa for helpful conversations and to  M. Jonsson,  Y. Gongyo and H.S. Kim for useful comments on earlier versions of this work.  The author is grateful to S. Boucksom who kindly brought a version of Lemma~\ref{boucksom} to his attention. This research was partially supported by SRC-GAIA through NRF Korea grant No.2011-0030795 and by Basic Science Research Program through NRF Korea funded by the Ministry of Education (2018R1D1A1B07049683).


\qa

\section{Singular hermitian metrics and plurisubharmonic functions}

  We refer to \cite{D11}, \cite{B20}, \cite[Appendix B]{BBJ21} for introduction to singular hermitian metrics of a line bundle and plurisubharmonic (i.e. psh) functions. 
We  have some \textbf{{conventions}} used in this paper.

\begin{itemize}

\item  We will often write holomorphic ($\QQ$-)line bundles additively as in $L_1 + L_2 := L_1 \otimes L_2$.  

\item

  We will often denote a singular hermitian metric $e^{-\vp}$ of a line bundle $L$  simply by $\vp$. For singular hermitian ($\QQ$-)line bundles $(L, \vp)$, we write additively  as in $(L_1 + L_2 , \vp_1 + \vp_2)$.

\end{itemize}

\subsection{Singular hermitian metrics on line bundles} \hfill \\

  A \emph{psh metric} of a holomorphic $\QQ$-line bundle $L$ on a complex manifold $X$  is a singular hermitian metric of $L$ with semipositive curvature current (so that its local weight functions can be taken as psh  functions), cf. \cite{D11}, \cite{HPS}. 
When a $\ZZ$-line bundle $L$ is given  transition functions $\{ g_{ij} \}$ on a locally trivializing open cover $\{ U_i \}_{i \in I}$,  a (smooth or singular) hermitian metric of $L$ can be identified with a collection of functions $e^{-\vp_i} = \abs{g_{ij}}^{-2} e^{-\vp_j}$. (A similar description holds when $L$ is a $\QQ$-line bundle by taking some multiple $mL$ ($m \ge 1$) that is a genuine holomorphic line bundle.)
 A holomorphic section $s \in H^0(X, L)$ defines a psh metric $\vp$ of $L$ by taking $\vp_i = \log \abs{s_i}^2$ (from $s_i = s_j g_{ij}$), which can be also denoted by $e^{-\vp} = \frac{1}{\abs{s}^2}$.  
 
 We recall the following fact due to \cite{D92} (cf. \cite[(3.5)]{FF17}). 
       
 \begin{proposition}\label{psa}
 
 Let $X$ be a compact complex manifold.  If $L$ admits a  psh metric with vanishing Lelong numbers, then it is nef.
 \end{proposition}   
\noi   Note that  the converse  does not hold, see  \cite[Example after (6.11)]{D11}.

\subsection{Valuative equivalence of plurisubharmonic singularities}  \hfill \\

 \noi A function on a complex manifold, with values in $\RR \cup \{ -\infty \}$ is \emph{quasi-psh} if it is locally the sum of a psh function and a smooth $\RR$-valued function. 
 If two quasi-psh functions (or two psh metrics) $\vp, \psi$ on a complex manifold  satisfy that $\vp - \psi$ is locally bounded above (near every point of the domain), i.e. $\vp \le \psi + O(1)$, we say $\psi$ is less singular than $\vp$, following \cite[(6.3)]{D11}.
If $\vp = \psi + O(1)$, 
  we say that $\vp$ and $\psi$ have \emph{equivalent singularities}.

As a flexible (and quite different) variant, we  will use the following notion (cf. \cite[(8.5)]{FJ05} when the dimension is $2$, cf. \cite{BFJ08}, \cite{KR22}).  
    
\begin{definition}\label{vequiv}
  We say that two psh functions $\vp$ and $\psi$ on a complex manifold $X$ are \textbf{\emph{valuatively equivalent}}  if the following two equivalent conditions (due to \cite{BFJ08} when combined with the strong openness  \cite{GZ15}) hold: 
 
 (1) For all real $m > 0$, the multiplier ideals are equal : $\JJ(m\vp) = \JJ(m\psi)$.

 (2) At every point of a proper modification $\pi: X' \to X$, the Lelong numbers of $\pi^* \vp$ and $\pi^* \psi$ coincide. Equivalently, for every divisorial valuation $v$ centered on $X$, we have $v(\vp) = v(\psi)$.

\end{definition}

 When two quasi-psh functions $\vp$ and $\psi$ on a complex manifold are valuatively equivalent, we note the following facts: 
 \begin{enumerate}

 \item
 
 If, moreover,  $\vp - \psi$ (or $\psi - \vp$) happens to be quasi-psh, it should have vanishing Lelong numbers. Of course, in general, none of $\vp-\psi$ and $\psi-\vp$ needs to be quasi-psh (even if one of $\vp, \psi$ has analytic singularities): see, for example, \cite[(2.3), (2.9)]{KS19}.

 \item
 If  $\psi$ has analytic singularities, then $\vp \le \psi + O(1)$, cf.  \cite[Thm. 4.3]{K15}.

 \end{enumerate}

See e.g. \cite[(2.8)]{KS19}, \cite[(4.7)]{KR22} for more examples and results on valuative equivalence of plurisubharmonic singularities.

   \subsection{Singular volume forms}  \hfill \\

 \noi  In this paper,   a \textit{singular volume form} $\alpha$ on a complex manifold $X$ is a continuous volume form that can be locally written as $\alpha(w) = g(w) \abs{dw_1 \wedge \ldots \wedge dw_n}^2$ in local analytic coordinates $w=(w_1, \ldots, w_n)$ where  $g \ge 0$ is a  continuous function with values in $\RR_{\ge 0} \cup \{ +\infty \}$. Here  $ \abs{dw_1 \wedge \ldots \wedge dw_n}^2$ denotes the  standard volume form $dx_1 \wedge dy_1 \wedge \ldots \wedge dx_n \wedge dy_n$ with $dw_j = dx_j + i dy_j$ (cf. \cite[(1.2)]{BJ17}). We also freely identify a singular volume form with the corresponding measure especially in the context of Ohsawa measures.

\begin{itemize}

\item
   We will say that a singular volume form $\alpha$ on $X$ has \textit{(divisorial) poles along an  snc $\QQ$-divisor} (not necessarily effective) $R  = \sum^m_{i=1} a_i R_i$ on $X$ if it can be written in local analytic coordinates on $V \subset X$ adapted to $R$ as 
 
\begin{equation}\label{volform}
 \alpha(w) = g(w) ( \prod^k_{i=1}  \abs{w_i}^{-2 a_i} )  \abs{dw_1 \wedge \ldots \wedge dw_n}^2  
\end{equation}
\noindent  where $R|_V =  \sum^k_{i=1} a_i R_i$,  $R_i = \divisor(w_i)$ and  $g > 0$ is a  locally bounded (above and below) real-valued continuous function (i.e. $g$ has neither zeros nor singularities). 

\item
Similarly, we will say that a real-valued function $t$ has \emph{(divisorial) poles along} such $R = \sum a_i R_i$ if it can be  written locally 
 
 \begin{equation}\label{singf}
 t(w) = g(w) ( \prod^k_{i=1}  \abs{w_i}^{-2 a_i} ). 
 \end{equation}

\end{itemize}

 For two continuous singular volumes $\al$ and $\beta$ on a complex manifold $X$, we will say that $\al$ is \emph{less singular than} $\beta$ if, whenever locally written in terms of local coordinates, $\alpha(w) = g(w) \abs{dw_1 \wedge \ldots \wedge dw_n}^2$, $\beta(w) = h(w) \abs{dw_1 \wedge \ldots \wedge dw_n}^2$ with density functions $g = e^{-u}, h= e^{-v}$, we have $u \ge v + O(1)$. (This terminology is compatible with the case when $u$ and $v$ are quasi-psh functions.) If $u = v + O(1)$, we will say that $\al$ and $\beta$ have equivalent singularities.

Now let $f: X_0 \to Y_0$ be a proper holomorphic submersion between complex manifolds with $\dim X = m+n$ and $\dim Y=m$. The operation of \emph{fiberwise integration} sends a singular volume form $\al$ on $X$ to a singular volume form $f_* \al$ on $Y$, whose definition and basic properties are reduced to the local product case by using a partition of unity (cf. \cite[Chapter 1, (2.15) Special case]{DX}, \cite{N07}).

\section{$L^2$ metrics}

 In \cite{PT} (cf. \cite{BP08}, \cite{HPS}), the authors studied semipositivity of direct images of the form $M := f_* (K_{X / Y} + L)$ for a surjective projective morphism $f: X \to Y$ in terms of certain naturally defined singular hermitian metrics on $M$. We will use the special case of \cite[(3.2.2)]{PT} when $M$ is torsion-free of rank $1$.  
 
  \subsection{Direct images of adjoint line bundles}  \hfill \\

\noi  In this paper, we  use the following special case of   \cite[Thm. 3.3.5]{PT} (cf. \cite[(21.2)]{HPS}).

\begin{theorem}[Positivity of direct images, a special case]\cite[Thm. 3.3.5]{PT}, cf. \cite[Set-up 3.2.1]{PT} \label{hps} 
 Let $f: X \to Y$ be a surjective projective morphism with connected fibers between two connected complex manifolds. Let $Y_0$ be the set of regular values of $f$.

 \begin{itemize}
 \item
 Let $L$ and $M$ be $\ZZ$-line bundles on $X$ and on $Y$ respectively, such that $K_X + L = f^* (K_Y + M)$ holds.
  
 \item Let $A$ be a $\ZZ$-line bundle on $X$ such that $f_* A$ is torsion-free of rank $1$. Let $Y_{1, \text{free}}$ be the Zariski open subset of $Y$ where $f_* A$ is locally free of rank $1$. 
 
 \item Let $(L\otimes A, \lambda)$ be a psh metric such that the inclusion 
 
 \begin{equation}\label{generic}
 f_* (K_{X/Y} \otimes L \otimes A \otimes \JJ(\ld) ) \to f_* (K_{X/Y} \otimes L  \otimes A)
 \end{equation}
 
\noi is generically an isomorphism. 

\end{itemize}
  Then the $L^2$ metric $\mu$ for the holomorphic line bundle given by the direct image $M \otimes f_* A = f_* (K_{X/Y} \otimes L \otimes A)$ restricted to $Y_0 \cap Y_{1, \text{free}}  $ is a psh metric.  The $L^2$ metric $\mu$ extends to a psh metric for the line bundle given by $M \otimes f_* A$ restricted to $Y_{1, \text{free}}$. 

\end{theorem}

We will recall the definition of the $L^2$ metric $\mu$ (cf. \cite[(3.2.2)]{PT}) in Definition~\ref{L2metric}. 
  We define the $L^2$ metric on $Y_0 \cap Y_{1, \text{free}}$.  For the simplicity of notation, we denote $Y := Y_0 \cap Y_{1, \text{free}}$ in the following paragraphs up to Definition~\ref{L2metric}.

 Let $u \in H^0 (Y, M \otimes f_* A)$ be a given section. Since $M \otimes f_* A = f_* (K_X + L  + A + f^*K_Y^{-1})$ (where the connected fibers assumption $f_* \OO_X = \OO_Y$ is used), we have $$H^0 (Y, M \otimes f_* A) = H^0 (Y, \Hom (K_Y, f_* (K_X + L+A))).$$  
 
 Hence viewing $u$ as a sheaf morphism $u : K_Y \to f_* (K_X + L+A)$, for a nowhere vanishing local section $\eta$ of  $K_Y$ on an open subset $V \subset Y$, we have $u(\eta) \in H^0 (f^{-1} (V), K_X + L+A)$. From \cite[(3.2.2)]{PT}, we have locally
 
 \begin{equation}\label{divide}
  u (\eta) = \sigma_{i} \wedge f^* \eta 
 \end{equation}
 
\noi for some  $\sigma_{i}$, an $(L+A)$-valued holomorphic $n$-form ($n = \dim X - \dim Y$) defined on $U_i$ from an open cover $\{ U_i \}_{i \in I}$ of $f^{-1} (V)$. The existence of such $\sigma_i$ follows from computation of elementary nature in terms of local coordinates (`admissible coordinates' \cite[\S2.2]{MT08}) which locally make $f$ a projection.  The restrictions $\sigma_{ i} |_{X_y}$ glue together to define $\sigma |_{X_y}$ in the integral below which defines the $L^2$ metric $\mu$. 
 
 \begin{definition}[$L^2$ metric on $Y_0  \cap Y_{1, \text{free}}$]\cite[(3.2.2)]{PT} \label{L2metric}
 Let $Y_0 (\subset Y)$ be the set of regular values of $f$. 
  Let $(L+A, \lambda)$ be the given hermitian line bundle in  Theorem~\ref{hps}. We define the $L^2$ metric  $\mu$ (induced from $\ld$) for the $\QQ$-line bundle $M \otimes f_* A$ on $Y := Y_0  \cap Y_{1, \text{free}}$  pointwise for $y \in Y$ by letting
\begin{equation}\label{induced}
  \left( \abs{u}^2 \cdot e^{-\mu} \right) (y) := \int_{X_y}  c_n \sigma |_{X_y} \wedge \overline{\sigma} |_{X_y} e^{-\ld}  =  \int_{X_y} \abs{\sigma_y}^2 e^{-\ld}
 \end{equation}

 \end{definition}
 
\noi where $u \in H^0 (Y, M \otimes f_* A)$ and $c_n = i^{n^2}$.  Here we write $\sigma_y :=  \sigma |_{X_y}$ noting that $\sigma$ itself may not be globally defined. Also note that the family $\sigma_y$ is determined by $u$ only.

 We recall the following fiberwise integration property of the $L^2$ metric from \cite{PT}, cf. \cite{HPS}.

\begin{proposition}\label{fiberwise}

Let $V \subset Y_0$ and $u(\eta) \in H^0 (f^{-1} (V), K_X + L+A)$ be as above. 
The volume form defined by  $\abs{u(\eta)}^2 e^{-\ld}$  on $f^{-1} (V)$ has its fiberwise integration along $f$ equal to $\abs{u \cdot \eta}^2 e^{-\mu}$ where $u \cdot \eta \in H^0 (V, K_Y + M )$.  

\end{proposition} 

\begin{proof} 
As in \eqref{divide}, we have $u(\eta) = \sigma \wedge f^* \eta$. Note that this last wedge product makes sense even though $\sigma$ alone may not be defined globally. 
For simplicity of notation, let us write $f: X \to Y$ in the place of $f: f^{-1} (V) \to V$.  Since $ \abs{ u(\eta) }^2$ is a $\abs{L+A}^2$-valued volume form on $X$ (where a section of $\abs{L+A}^2$ is viewed as a collection of local functions compatible with $g_{ij} \overline{g_{ij}}$ for transition functions $g_{ij}$ of $L+A$), its multiplication with a metric $e^{-\ld}$ for $L+A$ is defined and equal to

$$ \abs{ u(\eta) }^2 e^{-\ld} = \abs{\sigma \wedge f^* \eta}^2 e^{-\ld} = \abs{\sigma}^2 e^{-\ld} \cdot \abs{f^* \eta}^2   $$ which is a volume form, i.e. an $(m+n, m+n)$ form on $X$. By the projection formula of fiberwise integration \cite{DX} and using \eqref{induced}, the fiberwise integration of this along $f$ is equal to $  \abs{u}^2 e^{-\mu}. \abs{\eta}^2. $
\end{proof}

 The following example  illustrates Theorem~\ref{hps}. 

 \begin{example}\label{ruled}

(1)  Let  $f:  X \to Y$ is a ruled surface where $X = \PP (E)$ for a holomorphic vector bundle $E$ of rank $2$ on a smooth projective curve $Y$, cf. \cite[\S V.2]{H}.   When $L = O_{\PP (E)} (2)$ and $M = \det E$, we have $K_X + L = f^* (K_Y + M)$. When $L$ is ample and provided with a smooth psh metric, Theorem~\ref{hps} says that $M$ is semipositive, which is consistent with \cite[\S V, (2.20), (2.21)]{H}, characterization of ample line bundles on $X$. 
 
 (2)  Now let $X = \PP (E)$ be as in \cite[(1.7)]{DPS94} with $L = O_{\PP (E)} (2)$ : it is shown there that the only possible psh metrics $\ld$ for $L$ are ones whose curvature current is equal to $2[C]$ where $C \subset X$ is a section of $f$ (in particular, $L$ is not semipositive).  It is clear that this $\ld$ does not satisfy the condition \eqref{generic} of Theorem~\ref{hps}.

\end{example}

\subsection{Generalization to $\QQ$-line bundles}  \hfill\\

 In this subsection, we will generalize Theorem~\ref{hps} to Theorem~\ref{pt}.
 Let $f: X \to Y$ be as before (i.e. as in Theorem~\ref{pt}). 
 
  Let $L$ and $M$ be $\QQ$-line bundles satisfying the relation $K_X + L = f^* (K_Y + M)$. 
 Note that even when $K_X, L$ and $K_Y$ are $\ZZ$-line bundles, it is possible that $M$ is only a $\QQ$-line bundle (see e.g. \eqref{elliptic2}). 
 
 Recall that the only case of the direct images we use from \cite[Thm. 3.3.5]{PT} is of the form $M \otimes f_* A$  in Theorem~\ref{hps}, and we only use its restriction to $Y_{1, \text{free}}$ which is a line bundle. Recall that  $Y_{1, \text{free}}$ is the Zariski open subset of $Y$ where $f_* A$ is locally free of rank $1$.

  In the case of a $\QQ$-line bundle $M$,  we will use the notation $M \otimes f_* A$ only to denote the $\QQ$-line bundle $M \otimes (f_* A |_{Y_{1, \text{free}}})$ on $Y_{1, \text{free}}$. 
  We will also refer to it as the `$\QQ$-line bundle given by $M \otimes f_* A$ restricted to $Y_{1, \text{free}}$’. The $L^2$ metric will be defined as a singular hermitian metric for this $\QQ$-line bundle, as follows.

Let $N = (k-1)M$ where $k\ge1$ is the smallest integer such that $kM$ is a $\ZZ$-line bundle (i.e. a usual holomorphic line bundle). Then $K_X + L+ f^*N = f^* (K_Y + M + N)$ is a $\ZZ$-line bundle on $X$.

 Now let $(L+A, \ld)$ be a psh metric. 
 Let $(N, \chi)$ be a smooth hermitian metric (without requiring curvature condition), which plays an auxiliary role. 
  Equip $L + A + f^*N$ with the metric $\lambda + f^* \chi$. Applying Definition~\ref{L2metric} of the $L^2$ metric in the $\ZZ$-line bundle case, we get the $L^2$ metric $\mu_\chi$ for $(M+N) \otimes f_* A$, induced from $\ld + f^* \chi$. 
  
  \begin{definition}\label{L2metricQ}

   Multiplying $(-N, -\chi)$ to this, we get a metric $\mu_\chi - \chi$ for the $\QQ$-line bundle $M \otimes f_* A$ restricted to $Y_0  \cap Y_{1, \text{free}}$. 
   We define this $\mu_\chi - \chi$ to be the  \emph{\textbf{$L^2$ metric}} induced from the given singular hermitian metric $\ld$. 
  \end{definition}

This definition is justified by the following 
\begin{proposition}\label{well}

The $L^2$ metric $\mu := \mu_\chi - \chi$ for the $\QQ$-line bundle $M \otimes f_* A$ restricted to $Y_0  \cap Y_{1, \text{free}}$, is well-defined, independent of the choice of $\chi$. 

\end{proposition}

\begin{proof} 
 We need to check that $\mu_\chi - \chi$ is independent of the choice of the smooth metric $\chi$. It suffices to check this at each point $y \in Y_0  \cap Y_{1, \text{free}}$, hence we may fix a local trivialization of the holomorphic line bundle $kM$ in a neighborhood $V$ of $y$.  Then $\mu_\chi$ and $\chi$ are functions on $V$.
 Now let $u$ be a holomorphic section of $M + N = kM$.
From \eqref{induced}, we have

\begin{equation}\label{psipsi}
 \abs{u}^2 e^{-{\mu}_\chi } (y) = \int_{X_y} \abs{\sigma_y}^2 e^{-\lambda} e^{-f^* \chi} = e^{-\chi(y)}    \int_{X_y} \abs{\sigma_y}^2 e^{-\lambda} 
 \end{equation}
where $\sigma_y$ is as in Definition~\ref{L2metric}.   We get the second equality since $f^*\chi$ is constant on $X_y$. 
\end{proof}

Also using this definition involving $M+ N = kM$ and $\chi$, we have the fiberwise integration property for the $L^2$ metric inherited from Proposition~\ref{fiberwise}. 
 
Now we have  the following version of Theorem~\ref{hps} when $L$ and $M$ are $\QQ$-line bundles.

\begin{theorem}\label{pt}

 Let $f: X \to Y$ be a surjective projective morphism with connected fibers between two connected complex manifolds.  Let $Y_0$ be the set of regular values of $f$. 
 
 \begin{enumerate}
 \item
 Let $L$ and $M$ be $\QQ$-line bundles on $X$ and on $Y$ respectively, such that $K_X + L = f^* (K_Y + M)$ holds as an equality of $\QQ$-line bundles.

 \item Let $A$ be a $\ZZ$-line bundle on $X$ such that $f_* A$ is torsion-free of rank $1$. Let $Y_{1, \text{free}}$ be the Zariski open subset of $Y$ where $f_* A$ is locally free of rank $1$.

  \item Let $(L+A, \lambda)$ be a psh metric such that $\JJ(\ld|_F) = \OO_F$  for a general fiber $F$ of $f$.
  \end{enumerate}
   Then the $L^2$ metric $\mu$ for the $\QQ$-line bundle   $M \otimes f_* A$ restricted to $Y_0  \cap Y_{1, \text{free}}$  is a psh metric and it extends to a psh metric for the $\QQ$-line bundle  $M \otimes f_* A$ restricted to $Y_{1, \text{free}}$.

\end{theorem}

The case when $A = \OO_X$ is used in the proof of Theorem~\ref{main} while the general case is for the proof of Theorem~\ref{fi}.

\begin{proof}

Both assertions (being a psh metric and extension) are local properties. 
As before, let $N = (k-1)M$ where $k\ge1$ is the smallest integer such that $kM$ is a $\ZZ$-line bundle. For $y \in Y_0  \cap Y_{1, \text{free}}$,  fix a local trivialization of $kM$ in a neighborhood $V$ of $y$. Take $\chi = 0$ and apply Theorem~\ref{hps} to $f^{-1} (V) \to V$ to obtain the psh property. Similarly, extension to $Y_{1, \text{free}}$ follows from applying Theorem~\ref{hps} to $f^{-1} (V) \to V$ where $V$ is a neighborhood of points $y \in Y_{1, \text{free}} \setminus Y_0$. 
 \end{proof} 

\quad
\\

\section{The setting and the valuative equivalence result} 

 In this section, we give the proof of the main theorems.

 \subsection{The setting of Kawamata's canonical bundle formula}  \hfill \\
 
\noi We first give the setting of the main theorems, which is the same as in \cite{Ka98} (cf. \cite{Ko07}) except that we are in the generality of complex manifolds (where the necessary Hironaka resolution theorems also hold, cf. \cite{H64}, \cite{AHV77}, \cite{W09}). 

  Let $(X,R)$ and $(Y,B)$ be two pairs of complex manifolds and  snc $\QQ$-divisors. 
Let $f: X \to Y$ be a surjective projective morphism with connected fibers. An irreducible component $R_i$ of $R$ is called \emph{horizontal} if $f(R_i) = Y$. Otherwise it is called \emph{vertical}. We write $R = R_h + R_v$ where $R_h$ is the horizontal part and $R_v$ is the vertical part. 
 For a divisor $ R = \sum a_i R_i$ ($a_i \neq 0$), define $\red(R) := \sum R_i $.

\begin{definition}\label{snc}\cite[Def. 8.3.6]{Ko07}, \cite[Thm. 2]{Ka98}
We will say that $f: (X,R) \to (Y,B)$ satisfies {\bf{Kawamata's condition}} if the following hold: 

\begin{itemize}

\item $B = \sum B_i$ is a reduced snc divisor on $Y$, i.e. $B = \red (B)$. 

\item $\red(R) + f^* B$ is an snc divisor on $X$.

\item $f$ is a submersion over $Y \setminus B$. 

\item $f(\Supp (R_v)) \subset B$. 

\item $R_h$ is relative snc over $Y \setminus B$. 

\item Coefficients of $R_h$ are in the interval $(-\infty,1)$.

\item $\rank f_* \OO_X(  \lceil -R_h \rceil ) =1$. 

\item $K_X + R$ is $\QQ$-linearly equivalent to the pullback of some $\QQ$-Cartier divisor on $Y$. 
\end{itemize}

\end{definition} 
 
\noi  Let $B_R$ be the \textbf{discriminant divisor} induced by $R$ defined as  $B_R = \sum c_i B_i$ where we set
 
\begin{equation*}\label{discriminant} 
 c_i := 1- \sup \{ c: (X , R + c f^* B_i) \text{\; is log canonical over the general point of } B_i  \} 
\end{equation*} 
  (cf. \cite[Def. 3.1]{Am99}, also cf. \cite{Ka98}). The condition in the sup is that, in other words,  the non-lc locus of the snc divisor $R + c f^* B_i$ is contained in the inverse image of the union of some Zariski closed proper subsets of $B_i$'s. Since $\red (R) + f^*B$ is snc, the lc (resp. klt) conditions are determined by coefficients of all the components being $\le 1$ (resp. $<1$), cf. \cite[Cor. 3.12]{Ko97}. 
  
  Now write $R_v = R_{v1} + R_{v2}$ where the sum $R_{v1}$ (resp. $R_{v2}$) consists of components whose images under $f$ are of codimension $1$  (resp. of codimension at least $2$) in $Y$.  Since the coefficients of $R_h$ are less than $1$, the condition in \eqref{discriminant} is then equivalent to $(X, R_{v1} + c f^* B_i) \text{\; being lc over the general point of } B_i $. Hence $B_R$ is also characterized as the unique smallest $\QQ$-divisor supported on $B$ such that 
  
  \begin{equation}\label{disc}
  R_{v1} + f^* (B - B_R) \le \red (f^* B)
  \end{equation}
cf. \cite[(8.3.7) (2)]{Ko07}. \footnote{We thank Hyunsuk Kim for pointing out to replace $R_v$ by $R_{v1}$ in \eqref{disc}.}
    (Also see \cite[p.895]{Ka98} for the equivalent original definition of $B_R$.)

\begin{remark1}\label{Rv}
(1)  It can be easily seen (from local expression of $f$ in terms of coordinates adapted to the snc divisors) that the components of $R_{v1}$ are contained in the components of $f^* B$. 

(2) Although the discriminant divisor $B_R$ is not depending on $R_h$, we do not have much choice to choose $R_h$ due to the log Calabi-Yau condition (Definition~\ref{snc}, (9)) which is indeed used in Lemma~\ref{BR}. 

\end{remark1}

 Later we use the following lemma. 
 
\begin{lemma}\label{add S}

 In the setting of (\ref{snc}), let $S$ be a $\QQ$-divisor supported on $B$. The discriminant divisor $B_{R + f^* S}$ associated to $R + f^* S$ is equal to $B_R + S$. 

\end{lemma}

\begin{proof}

It follows from the definition putting $R' = R + f^*S$ and $R'_v = R_v + f^*S$. 
\end{proof}

  In the setting of Definition~\ref{snc},  a $\QQ$-line bundle $ J(X \slash Y, R)$ (called the \emph{moduli part})  is defined on $Y$ by the following relation  (cf. \cite{Ka98})
  
\begin{equation}\label{Jpart}  
  K_X + R \sim_{\QQ} f^* ( K_Y + B_R +  J(X \slash Y, R)). 
\end{equation} 

\noi For the purpose of this paper involving metrics, we also  view \eqref{Jpart} as the corresponding equality of $\QQ$-line bundles
 
\begin{equation}\label{relation2}
 K_X +  L = f^* ( K_Y + M) = f^* (K_Y + J + H)
\end{equation} 
 where we define the corresponding $\QQ$ line bundles  $L := \OO(R), H := \OO(B_R), \; J := J(X/Y, R)$ and  $M := J+H$. 

Also we record the following obvious statement for later use. 
 
\begin{lemma}\label{Rh}

 Let $R_h = \sum a_i D_i$ be the snc divisor in Definition~\ref{snc} with the coefficients $a_i$ in the interval $(-\infty, 1)$. Then the divisor $R_h +  \lceil -R_h \rceil $ has all the coefficients in the interval $[0, 1)$. 

\end{lemma}

Let $A = \OO( \lceil -R_h \rceil)$ be the corresponding $\ZZ$-line bundle on $X$. Note that $A = \OO_X$ in the setting of Theorem~\ref{main}. 
 
 \subsection{Valuative equivalence of the $L^2$ metric}  \hfill\\

\noi In the setting of the previous subsection, the following is the key  result in the proof of the main results.

\begin{theorem}\label{valu}

 Let $f: X \to Y$,  $K_X + L = f^* (K_Y + M)$ and $A= \OO( \lceil -R_h \rceil)$ be as in the previous subsection.   Assume that the snc divisor  $R_{v1}$ (i.e. those vertical components whose images are of codimension $1$ in $Y$) is effective. (Hence the discriminant divisor $B_R$ is effective.)
 
 \begin{itemize}
 \item
 Let $\ld$ be a singular hermitian metric of $L+A$ given by the divisor $R +  \lceil -R_h \rceil$ (so that $\JJ(\ld|_F) = \OO_F$ holds  for a general fiber $F$ of $f$ due to Definition~\ref{snc} (6), (7)). 
 
 \item Let $\mu$ be the $L^2$ metric (induced from $\ld$) for $M \otimes f_* A$, defined in Definition~\ref{L2metricQ}. 
  \end{itemize}

\noi   Then the following hold. 

\begin{enumerate}
\item The  $L^2$ metric $\mu$ is a psh metric  for the $\QQ$-line bundle   $M \otimes f_* A$ restricted to the Zariski open subset  $Y_{1, \text{free}} \subset Y$.

\item
Let $\vp_{B_R}$ be a psh metric for  the $\QQ$-line bundle $\OO(B_R)$ given by the $\QQ$-divisor $B_R \ge 0$. 
The $L^2$ metric $\mu$ is valuatively equivalent to $\vp_{B_R}$ on $Y_{1, \text{free}}$. 
\end{enumerate} 

 \end{theorem}

 \begin{proof}

 (1) 
 Let $Y' := Y \setminus Z$ where $Z$ is the image of the components in $R_{v2}$ (hence $Z$ is of codimension $\ge 2$). 
  Apply Theorem~\ref{pt} to the restriction $f^{-1} (Y') \to Y'$ noting that the condition (2) of Theorem~\ref{pt}  is satisfied for $A= \OO( \lceil -R_h \rceil)$, cf. \cite{Ka98}, \cite{Am04}. 
  
  Note that $\ld$ given by (i.e. having divisorial poles along) the divisor  $R +  \lceil -R_h \rceil = R_h + \lceil -R_h \rceil + R_{v1} + R_{v2}$ is a psh metric on $f^{-1} (Y')$ since it is taking the complement of the possibly noneffective  divisor $R_{v2}$.  Hence the $L^2$ metric $\mu$ is a psh metric on $Y_{1, \text{free}} \setminus Z$ by Theorem~\ref{pt}. Since $Z$ is of codimension $\ge 2$, $\mu$ extends to $Y_{1, \text{free}}$ as a psh metric of the same $\QQ$-line bundle (as its psh local weight functions extend so). 
 \\
 
 (2) 
 Let $\tilde{R} := R +  \lceil -R_h \rceil$. Note that, from the definition, $B_{\tilde{R}} = B_R$. In the rest of this proof, for the convenience of notation, we will use $R$ to denote $\tilde{R}$. 
 
We need to show that $v(\mu) = v(B_R) (:= v(\vp_{B_R}))$ for every divisorial valuation $v = \ord_G$ where $G$ is a prime divisor over $Y$ with nonempty center in $Y$, cf. \cite[Section B.5]{BBJ21}. More precisely, there exists a proper modification  $\pi : Y' \to Y$ such that $G \subset Y'$ is a prime divisor and its center on $Y$ is $\pi(G) \neq \emptyset$. 

Now we consider the following diagram \eqref{XYs} where $\rho: X' \to X$ is bimeromorphic and $f'$ also satisfies Kawamata's condition (\ref{snc}). (The morphism $f': X' \to Y'$ in this diagram is obtained by first taking the fibration $X'' \to Y'$ induced by the bimeromorphic base change $Y' \to Y$ and then taking further blow-ups over $X''$ : cf. \cite[8.4.8]{Ko07}.)

   We take a Zariski open subset $Y_0 \subset Y$ and let $X_0 := f^{-1} (Y_0)$ so that $f: X_0 \to Y_0$ is submersion. Also determine $Y'_0$ and $X'_0 := f'^{-1} (Y'_0)$  in the following diagram so that they are isomorphic to $Y_0$ under $\pi$ and to $X_0$ under $\rho$ respectively.

\begin{equation}\label{XYs}
\begin{CD}
 X_0' \subset X' @>\rho>> X \supset X_0 \\
 @VVf'V  @VVfV \\
 Y_0' \subset Y' @>\pi>> Y \supset Y_0
\end{CD}
\end{equation}

  The restrictions $f: X_0 \to Y_0$ and $f' : X'_0 \to Y'_0$ can be identified with each other and so are the fiberwise integrations taken along them. This can be expressed as 
  
\begin{equation}\label{proj1}
  \pi^* (f_* u ) = f'_* (\rho^* u) 
\end{equation}  
  
\noi for a singular volume form $u$, where the domains are all restricted to $X_0, Y_0, X'_0, Y'_0$. Now let $R'$ be the divisor on $X'$ defined by $K_{X'}  + R' = \rho^* (K_X + R)$. Now we take $u$ to be a singular volume form with poles along $R$, so that the pullback $\rho^* u$ has poles along $R' = \rho^* R - K_{X' \slash X} $. 
  
  Consider the following projection formula of fiberwise integration combined with \eqref{proj1}: 
  
\begin{equation}\label{proj2}
 f'_* ( \rho^* u \wedge f'^* t) = f'_* (\rho^* u) \wedge t = \pi^* (f_* u ) \wedge t .
\end{equation}
  
  In order to show  $v(\mu) = v(B_R)$,  we will apply \eqref{proj2} for 
  the fiberwise integration taken along the restriction of $f'$ over a neighborhood $U \subset Y'$ of a general point of $G$ taking $t$ with poles $\delta B' - (v(\mu) - \al) G$ first and then with poles $\delta B' - (v({B_R})- \al)G$ secondly. Here $\al := \ord_G (K_{Y' / Y})$, i.e. the coefficient of the prime divisor $G$ in the relative canonical divisor $K_{Y' / Y}$.

 For arbitrary $\delta < 1$,  take a function $t$ with poles $\delta B' - (v(\mu) - \al) G$.  We claim that 
   
\begin{equation}\label{proj3}
    R'_{v1} + f'^* (\delta B' - (v(\mu) - \al)G) < \red (f'^* B'). 
 \end{equation} 

 To verify the claim \eqref{proj3}, first note that since $f_* u$ has the poles given by the psh weight $e^{-\mu}$, its pullback $ \pi^* (f_* u )$ on $Y'$ has poles given by the pullback of $e^{-\mu}$ divided by the contribution of the jacobian of the morphism $Y' \to Y$. When $g=0$ is a local equation of the divisor $G$, this pole is expressed as $e^{-\pi^* \mu} \abs{g}^{2 \al}$. 
 
 Since we are looking at a neighborhood $U$ of a general point of $G$ and considering the restriction of the divisor $\delta B' - (v(\mu) - \al) G$ to $U$, which is nothing but $(\delta - v(\mu) + \al)G$, the local expression of the pole  of $\pi^* (f_* u ) \wedge t$ in \eqref{proj2} is now given by 
 
 $$ e^{-\pi^* \mu} \abs{g}^{2 \al} \abs{g}^{-2 (\delta - v(\mu) + \al)} = e^{-\pi^* \mu} \abs{g}^{2 (v(\mu) - \delta)}. $$
 
 By Lemma~\ref{boucksom} applied for $W = Y', \vp = \pi^* \mu, \psi = 2(v(\mu)-\delta) \log \abs{g}$, this is locally integrable. Thus from the LHS of \eqref{proj2}, the singular volume form inside $f'_*$ is also locally integrable, which gives \eqref{proj3}. 
 
 Now from \eqref{proj3}, it follows that $R'_{v1} + f'^* (B' - (v(\mu) - \al)G) \le \red (f'^* B')$. From the definition of the discriminant divisor $B'_{R'}$ associated to $R'$ (see \eqref{disc}), we see that $v(B'_{R'})$ is the smallest possible coefficient for $G$ (in the place of ($v(\mu) - \al$)) to make this inequality hold. Hence we have

 \begin{equation}\label{proj4}
  v(\mu) - \al \ge v(B'_{R'}) = v(B_R) - \al 
  \end{equation}
 where the equality is from Lemma~\ref{BR}. Thus we have $v(\mu) \ge v(B_R)$.

  Now suppose that $v(\mu) - v(B_R) > 0$. This time, take a function $t$ to be with poles  $\delta B' - (v(B_R) - \al) G$ for $\delta < 1$. From \eqref{proj2}, we will have contradiction since the LHS is locally integrable while the RHS is not. Indeed, the RHS of \eqref{proj2} has poles along $$(v(\mu)- \al)G + \delta B' - (v(B_R) - \al) G  = (v(\mu) - v(B_R) )G + \delta B' $$ and it is not locally integrable when $v(\mu) - v(B_R) + \delta > 1$ (note that $G$ appears in $B'$ as we may assume so).

 On the other hand, the LHS of \eqref{proj2} is locally integrable since   $\rho^* u \wedge f'^* w$ is locally integrable before taking the fiberwise integration : it has poles along the snc divisor $T:= R' + f'^* ( \delta B' - (v(B'_{R'}) G)$ which is klt (since $\delta < 1$) for the following reason. Since $T$ is snc, we only need to check the coefficients of the components. For the vertical part of $T$ (with respect to $f'$), this follows   from the definition of the discriminant divisor $B'_{R'}$.
 
 Each horizontal component of $T$ is contained in $R'$ and is the strict transform under $\rho: X' \to X$ of a horizontal component for $f$ from the construction of $X'$. In view of $K_{X'} + R' - \rho^* R_v = \rho^* (K_X+ R_h)$, the coefficients of the horizontal part of $T$ is also less than $1$. 
  Note that our notation $R$  really means $\tilde{R} = R +  \lceil -R_h \rceil$ (for the given $R$ in the statement), where $\tilde{R}_h = R_h + \lceil -R_h \rceil$ is klt due to Lemma~\ref{Rh}.

  This is contradiction. Hence we have $v (\mu) = v( B_R)$, which completes the proof of Theorem~\ref{valu}. 
\end{proof} 

\begin{remark1}\label{ceiling}

 At the end of this argument, we see that adding precisely $S := \lceil -R_h \rceil$ to $R_h$ is essential for this result and its proof : $S$ must be an integral divisor (so that we can take the line bundle $A = \OO(S)$) and $S + R_h$ should be effective. 

\end{remark1}

 \begin{lemma}\label{BR}
 
  For $f$ and $f'$ in \eqref{XYs} satisfying Kawamata's condition, we have $v(B'_{R'}) = v(B_R) - \al$. 
 
  \end{lemma}
  
  \begin{proof}
  
  This follows from the equality of divisors $K_{Y'} + B'_{R'} = \pi^* (K_Y + B_R)$ which in turn follows immediately from the pull back property of the moduli part line bundles $J(X'/Y', R') = \pi^* J(X/Y, R)$ \cite[8.4.9 (3)]{Ko07}  (originally due to \cite{Ka98}). The pull back property results from the Hodge theoretic characterization of the moduli part line bundle and the fact that the canonical extension in Hodge theory commutes with pull backs. 
  \end{proof}

 Also the following lemma was used,  a weaker variant (using only one divisorial valuation) of valuative characterization of multiplier ideals (cf. \cite{BFJ08}, \cite{BBJ21}).

 \begin{lemma}\label{boucksom}
 
 Let $W$ be a complex manifold and $H$ a prime divisor on $W$. Let $\vp$ and $\psi$ be psh functions on $W$. Assume that $\psi$ has analytic singularities. If $\ord_H (\psi) > \ord_H (\vp) -1$, then $e^{\psi - \vp}$ is locally integrable at a general point of $H$. 
 \end{lemma} 
 
Note the minor point that here  $e^{\psi - \vp}$ is locally $L^1$ instead of $L^2$ as in \cite[Thm. 10.11]{B20} since our convention from \S 2.1 is taking, for example, $\vp = \log \abs{f}^2$ instead of $\vp = \log \abs{f}$ throughout this paper.
 
\begin{proof}

(1) First assume that  $\vp$ also has analytic singularities. Let $\mu: W' \to W$ be a log resolution of both $\vp$ and $\psi$. We may choose a general point $p$ of $H$ and  shrink $W$ to some $U$, a neighborhood of $p$,  in order to avoid prime divisors on $W$ along which $\vp$ or $\psi$ has positive generic Lelong numbers and also to avoid the images of exceptional divisors of $\mu$ along which $\mu^* \vp$ or $\mu^* \psi$ has positive generic Lelong numbers. Then it is easy to see from $\mu: \mu^{-1} (U) \to U$ that the condition $v(\psi - \vp) > -1$ for $v = \ord_H$ alone determines local integrability of $e^{\psi - \vp}$.

(2) Now let $\vp$ be a general psh function. Let $(\vp_m)_{m \ge1}$ be a Demailly approximation sequence of $\vp$. (cf. \cite[Thm. 7]{D13} for the case of a bounded pseudoconvex domain in $\CC^n$. In general, one can use partitions of unity as in \cite{D11} to choose such a sequence.) Since $v(\psi) > v(\vp) -1$, we can fix $\ld > 1$ such that $v(\psi) > \ld v(\vp) -1$. Then for every $m \ge 1$, we have $v(\psi) > \ld v(\vp_m) -1$ since $\vp_m$ is less singular than $\vp$ in the sense that $\vp_m \ge \vp + O(1)$.  Applying (1) to $\vp_m$ (having analytic singularities), we then have $e^{\psi - \ld \vp_m} \in L^1_{\loc, x}$ at a general point $x$ of $H$, i.e. for $x \in H \setminus V_m$ where $V_m \subset H$ is a Zariski closed subset (which may depend on $m$). By Lemma~\ref{D13}, we have $e^{\psi - \vp}  \in L^1_{\loc, x}$ for the same $x$. 
\end{proof}

\begin{lemma}\label{D13}\cite{D13}
Let $\psi, \vp$ be psh functions on a complex manifold $X$. Let $(\vp_m)_{m \ge 1}$ be a Demailly approximation sequence of $\vp$. Let $\ld > 1$ be a real number. For every $x \in X$ and every integer $m \ge \lceil \frac{\ld }{2(\ld -1)} \rceil $, we have the implication 

$$   e^{\psi - \ld \vp_m} \in L^1_{\loc, x} \implies e^{\psi - \vp}  \in L^1_{\loc, x} .$$
\end{lemma} 
 
\begin{proof} 

 When $\psi = \log \abs{f}^2$ for a holomorphic function $f$, this is \cite[Cor. 4]{D13} which states the inclusion of multiplier ideals $\JJ(\ld \vp_m) \subset \JJ(\vp)$.  The same argument (which uses  \cite[Lem. 2]{D13}) works when we put  $e^{\psi}$ in the place of $\abs{f}^2$. 
\end{proof}

\section{Proofs of the main theorems}  

 Now using Theorem~\ref{valu}, we complete the proof of Theorem~\ref{main}.

\begin{proof}[{\textbf{Proof of Theorem~\ref{main}}}]

   First note that $R_v$ is not necessarily effective.   By Remark~\ref{Rv},  there exists an effective divisor $S$ supported on $B$ such that $R_{v1} + f^* S \ge 0$ and $B_R + S \ge 0$.  Let $N = \OO(S)$ be the associated $\QQ$-line bundle.  Consider the equality of $\QQ$-line bundles
 
 $$ K_X + L + f^* N = f^* (K_Y + J + H + N) .$$
 
 Equip $L + f^* N$ with a (not necessarily psh) singular hermitian metric $\ld$ given by  (i.e. having divisorial poles along) the divisor $R+ f^* S$. Then the induced  $L^2$ metric $\mu$ for $J + H + N$ is a psh metric by Theorem~\ref{valu} (1), which  we are applying to the snc divisor $R+ f^* S = R_h + R_{v1} + R_{v2} + f^*S$ in the place of `$R$ in the statement of Theorem~\ref{valu}' 
 in view of $R_{v1} + f^* S \ge 0$.  Here note that the condition $\JJ(\ld|_F) = \OO_F$ for a general fiber $F$  in Theorem~\ref{pt} is satisfied since the coefficients of the horizontal snc divisor $R_h$ are assumed to be less than $1$.  
 
 Now consider the Siu decomposition \cite{Si74}, cf. \cite[(2.18)]{D11}, \cite[2.2.1]{B04}  of the curvature current $\Theta_{\mu}$ of $\mu$: 
 
 $$ T := \Theta_{\mu} = \sum_W \nu(T, W) [W] + R_T $$
where $W$ runs through all codimension $1$ irreducible analytic subsets of $X$ and $\nu(T, W)$ is the generic Lelong number of $T$ along $W$. The closed semipositive current $R_T$ is the residual part of the decomposition whose Lelong superlevel sets $E_c (R_T)$ have codimension $\ge 2$ for every $c > 0$.

 By Theorem~\ref{valu} (2), the current   $\Theta_{\mu}$ (or its psh potential) is valuatively equivalent to the current  given by the effective discriminant divisor $B_{R+ f^* S} = B_R + S \ge 0$ (by Lemma~\ref{add S}). Thus the divisor part  $\sum \nu(T, Y_k) [Y_k]$ is a finite sum which is precisely given by the discriminant divisor $B_{R + f^*S}$ of the snc divisor $R+ f^*S$. 
 
 We apply Lemma~\ref{decom} to the curvature current $T$ of  $(J+H+N, \mu)$ and the curvature current $Q$ of $(H+N, \vp_{B_R + S})$ where $\vp_{B_R +S}$ is a psh metric given by the effective divisor $B_R + S$. Since the closed positive $(1,1)$ current $R_Q = 0$, the current $R_T$ has zero Lelong numbers at every point by Lemma~\ref{decom}.

 Since the closed positive $(1,1)$ current $R_T$ belongs to the first Chern class of the $\QQ$-line bundle $J + H + N - (H+N) = J$, there exists a singular hermitian metric $\psi$ of $J$ whose curvature current is equal to $R_T$ (as is well-known, see e.g. \cite[p.50]{B04}). This $\psi$ is the psh metric we wanted in the statement of (1) of the theorem : it has vanishing Lelong numbers.

 Choose a singular hermitian metric $\vp_S$ given by the divisor $S$ such that $\vp_{B_R +S} = \vp_{B_R} + \vp_S$. From $(J + H + N, \mu)$, we subtract $(J, \psi)$ and get a psh metric $(H+N, \mu - \psi)$ given by the effective divisor $B_R + S$. Now subtracting again $(N, \vp_S)$, we get $(H, \eta := \mu - \psi - \vp_S)$ which is a singular hermitian metric given by the original discriminant divisor $B_R$. Since $B_R$ may not be effective, $\eta$ may not be a psh metric. This $\eta$ is the one we were looking for, which completes the proof of Theorem~\ref{main}. 
 \end{proof}

\begin{lemma}\label{decom}

 Let $X$ be a complex manifold.  Let $T$ and $Q$ be closed semipositive $(1,1)$ currents on $X$.  Suppose that $T$ and $Q$ are valuatively equivalent, i.e. they have the same Lelong numbers at every point in $X$ and at every point in all proper modifications $\tilde{X} \to X$.  Then in the Siu decomposition of the closed positive $(1,1)$ currents $T$ and $Q$, 
 
\begin{align*}
  T = \sum_W \nu(T, W)[W] + R_T \; \; \; \text{and} \; \; \;
  Q = \sum_W \nu(Q, W)[W] + R_Q, 
\end{align*}

\noi the two residual parts, i.e. the closed semipositive $(1,1)$ currents $R_T$ and $R_Q$ are valuatively equivalent. 

\end{lemma}

\begin{proof} 

  This is immediate from the construction of $R_T$ (and $R_Q$) of the Siu decomposition (cf.  \cite[2.2.1]{B04}) as the limit $R_T = T - \sum_{k \ge 1} \nu(T, W_k) [W_k]$ where $(W_k)_{k \ge 1}$ is an at most countable family of irreducible analytic subsets of codimension $1$.  Note that we have $\nu(T, W) = \nu (Q, W)$ for every $W$.  Also note that  $v(R_T) = v(R_Q)$ is nonzero only possibly for $v = \ord_G$ where $G$ is a prime divisor  over $X$ (i.e. $G \subset X' \to X$ for a proper modification $\pi: X' \to X$) whose center on $X$, $\pi(G)$ is of codimension  $\ge 2$ in $X$.  
  \end{proof}

Now we turn to the proof of Theorem~\ref{fi}.

\begin{proof}[\textbf{{{Proof of Theorem~\ref{fi}}}}]

 We will use the full setting of Theorem~\ref{valu} to the given $f: (X, R) \to (Y, B)$,  taking $A = \OO (\lceil -R_h \rceil)$.  Let $Y_1 := Y_{1, \text{free}}$ be the Zariski open subset of $Y$ where $f_* A$ is locally free of rank $1$. The result will follow from a version of Theorem~\ref{main} for the morphism $f: f^{-1} (Y_1) \to Y_1$. 
 \\
 
  \emph{\textbf{Step 1}}. Since $\al$ is a continuous singular volume form, so is its fiberwise integration $f_* \al$ on $Y_0$. Hence we may check `less singular' on $Y_0 \cap Y_1$ once $\theta$ is constructed from the $L^2$ metric in this case, which will be a singular hermitian metric defined on $Y_1$. 
  
  Since being less singular is a local property on $Y_0 \cap Y_1$, we may restrict to an open neighborhood of a point $p \in Y_0 \cap Y_1$. (The $\theta$ we consider will extend from $Y_1$ to $Y$, then we  use a partition of unity to obtain $\theta$ on $Y$ as in the statement.) 
  \\

  \emph{\textbf{Step 2}}.
  Now we take the following variant of Theorem~\ref{main} for $f: f^{-1} (Y_1) \to Y_1$. 
  Similarly to Theorem~\ref{main}, we apply Theorem~\ref{valu} putting $R + f^*S$ in the place of \emph{`$R$ of Theorem~\ref{valu}'} where $S \ge 0$ is an snc divisor supported on $B$ such that $R_{v1} + f^*S$ and $B_R + S \ge 0$. Let $N = O(S)$ be the associated $\QQ$-line bundle.

  We take the $L^2$ metric $\mu$ for $M \otimes f_* A$ (on $Y_{1, \text{free}}$)  induced from a psh metric $\ld$ for the $\QQ$-line bundle $L+f^*N  +A$ given by the effective divisor $$R +f^*S + \lceil -R_h \rceil = R_h +\lceil -R_h \rceil  + R_{v1} + f^*S + R_{v2}.$$

 Now Theorem~\ref{valu} (2) says that $\mu$ and $\vp_{B_R + S}$ are valuatively equivalent where $\vp_{B_R +S}$ has divisorial poles along the snc divisor $B_R + S$. 
 Take the curvature currents: let $T$ be the curvature current of the psh metric (i.e. the singular hermitian line bundle) $(J+ H+ N, \mu)$ and $Q$ the curvature current of the psh metric $(H+N, \vp_{B_R + S})$. From the Siu decomposition of $T$, we have $R_T$ with vanishing Lelong numbers, belonging to the first Chern class of the $\QQ$-line bundle $J$. This gives rise to a psh metric $\psi$ with vanishing Lelong numbers for $J$. Subtracting $(J, \psi)$ from $(J+H+N, \mu)$, we obtain a psh metric $(H+N, \mu - \psi)$ which is one given by the effective divisor $B_R + S$.

 Choose a singular hermitian metric $\vp_S$ given by the divisor $S$. Subtracting again $(N, \vp_S)$, we get $(H, \vp_{B_R} = \mu - \psi - \vp_S)$ which is a singular hermitian metric given by the original discriminant divisor $B_R$. Here we can and do arrange so that $\vp_{B_R +S} = \vp_{B_R} + \vp_S$. 
 \\

  \emph{\textbf{Step 3}}.
  Finally we use the fiberwise integration property of the $L^2$ metric $\mu$ on $Y_1$ given in Proposition~\ref{fiberwise} which continues to hold in the $\QQ$-line bundle setting of Proposition~\ref{well}. 
  
 Recall that in the notation of Proposition~\ref{fiberwise}, $u$ and $\eta$ are local sections so that  $\abs{u(\eta)}^2$ and $\abs{u \cdot \eta}^2$ are (local) smooth volume forms.    The fiberwise integration of the singular volume form $\abs{u(\eta)}^2 e^{-\ld}$ is equal to $\abs{u \cdot \eta}^2 e^{-\mu}$.  Hence the  fiberwise integration of $\tilde{\beta} := \abs{u(\eta)}^2 e^{-(\ld-f^* \vp_S)}$ is equal to $\abs{u \cdot \eta}^2 e^{-(\mu-\vp_S)}$, which we take as $\theta$. 
 Note that $\tilde{\beta}$ has divisorial poles given by $e^{-(\ld-f^* \vp_S)}$ which is along the divisor $R +  \lceil -R_h \rceil$. It has equivalent singularities with the given $\beta$. 
 
 Since $\al$ has divisorial poles along the divisor $R$, we have $\al$ less singular than $\tilde{\beta}$ and thus $f_* \al$ less singular than $f_* \tilde{\beta} = \theta$. 
  \end{proof}

\begin{remark1} \label{mild}
In Theorem~\ref{fi}, we have `less singular' instead of `equivalent singularities' since a component of $R_h$ may be possibly not transversal with a special smooth fiber of $F$ lying on a point of $B$. (Note that $R_h$ is relative snc over $Y \setminus B$, cf. Definition~\ref{snc}.) Even though the coefficients of $T:= R_h + \lceil -R_h \rceil$ are less than $1$, $(R_h +\lceil -R_h \rceil) |_F$ may not have such coefficients (for example, in local coordinates $(x,y)$, when $F$ is $(x=0)$ and a component of $T$ is $c (x-y^2)$ with $\frac{1}{2} \le c < 1$)  and may possibly contribute to the difference of singularities between $f_* \al$ and $f_* \beta$. 

 We note that, however, such contribution (and the resulting difference away from `equivalent singularities') is only mild in that it is still a locally integrable singularity on $Y$  due to the fact that both $R_h$ and $R_h + \lceil -R_h \rceil$ are klt divisors on $X$. 
\end{remark1}

\begin{remark1}
 In the case  when $\dim Y = 1$, the volume asymptotics from Theorems~\ref{main}, \ref{fi} can be compared with  many previous results such as  \cite{BJ17},   \cite[(2.1)]{EFM} (see also related works by 
  \cite{Y10},  \cite[(2.1)]{GTZ16}, \cite{GTZ19}, \cite[(3.8)]{Be16} and others). 
   Such volume asymptotics for $\int_{X_t} \abs{\sigma_t}^2$  can be roughly  summarized as 
    $ \abs{t}^{-2\al} \abs{\log  \abs{t}^2}^{\beta}$
\noi for certain $\al \ge 0$ and $\beta$ for a local analytic coordinate $t$ on $Y$ with $\dim Y = 1$. We refer to the above papers for precise statements.
\end{remark1}

 \quad
 \\

\section{Applications to $L^2$ extension theorems}

 Let  $(X, \psi)$ be a pair where $X$ is a complex manifold and $\psi$ is a psh metric with neat analytic singularities for a $\QQ$-line bundle $L$ on $X$. 
  A \emph{log canonical center} (or an \emph{lc center}) of $(X,\psi)$ is an irreducible subvariety $Y \subset X$ that is the image of a prime divisor $E$ in a log resolution $X' \to X$ of the pair, with its discrepancy $a(E, X, \psi)$ equal to $-1$ (i.e. the lowest possible value for an lc pair), cf. \cite[Def. 3.3]{Ko97}, \cite[Def. 2.3]{K21}. We  call such $E$ a \emph{log canonical place} (or an \emph{lc place}) of $Y$.  When $Y$ is an lc center and there is no other lc center $Y_1$ such that $ Y_1 \supsetneq Y $, we call  $Y$ a \emph{maximal lc center}.  The maximal lc centers are precisely the irreducible components of the non-klt locus of $(X, \psi)$. (We refer to \cite[Section 2]{K21} for more details on this setting.)  
 
\subsection{Ohsawa measure}  \hfill\\

\noi  Now let $Y$ be a maximal lc center of an lc pair $(X, \psi)$ as above. Let $\Psi$ be a quasi-psh function on $X$ with analytic singularities determined by the relation $e^{-\Psi} h = e^{-\psi}$ where $h$ is a smooth hermitian metric of $L$. In this setting, the Ohsawa measure $dV[\Psi]_Y$ is defined following the original definition of \cite{O01}, cf. \cite{D15}. 
   
\begin{definition}   \cite[Section 2]{D15} cf. \cite[Def. 3.1]{K21} \label{Omeasure}
Let $dV_X$ be a smooth volume form on $X$.   The  \textbf{ \emph{Ohsawa measure}} $dV[\Psi]_Y$ of $\Psi$ on $Y$ (with respect to $dV_X$) is 
a positive measure $d\mu$ on $\yr$ satisfying the following condition:  for every $g$, a real-valued compactly supported continuous function on $\yr$ and for every $\tilde{g}$, a compactly supported extension of $g$ to $X$, we have the relation 
 
 \begin{equation}\label{dvp}
 \int_{\yr} g \; d\mu = \lim_{t \to -\infty} \int_{\{x \in X, \; t < \Psi(x) < t+1 \}} \tilde{g} e^{-\Psi} dV_{X}. 
\end{equation}

\end{definition}

    See also \cite[Prop. 3.2]{K21} for the existence of this measure. In the notation $dV[\Psi]_Y$, its dependence on the given $dV_X$ is suppressed.

       Now assume that $Y$ has a unique lc place $E \to Y$ where $E \subset X'$ is on a log resolution $\mu: X' \to X$ of $(X, \psi)$.  Then the morphism $E \to Y$ has connected fibers (cf. \cite[Cor. 2.12]{K21}). Writing the discrepancy relation as in \cite[(5) in Proof of Prop. 3.2]{K21}, we have 
    
     \begin{equation}\label{discrep2}
 K_{X'} + E + F \equiv f^* (K_X + [\Psi]). 
 \end{equation}
 
 By  \cite[Prop. 3.2]{K21}, the Ohsawa measure $dV[\Psi]_Y$ is equal to the direct image $(\mu|_E)_* d\nu$ where $d\nu$ is a measure corresponding to a singular volume form with poles along the snc divisor $R := F|_E$. Hence $dV[\Psi]_Y$ is also a continuous singular volume form on $\yr$. 
 
 Following \cite{Ka98}, we may assume (by choosing $\mu: X' \to X$) that 
   the restriction $\mu|_E: E \to Y$ factors through a proper modification $\pi : Y' \to Y$  and  $f: (E, R) \to (Y', B)$ satisfies Kawamata's condition of Definition~\ref{snc} for a suitable reduced snc divisor $B$ on $Y'$. We have the following diagram. 
   
 \begin{equation}\label{EYX}  
  \xymatrix{  E \; \ar[d]^{f} \ar@{^{(}->}[r] & X' \ar[dd]^{\mu}  \\
   Y'  \ar[d]^{\pi}   & \\
    Y \ar@{^{(}->}[r] & X }
    \end{equation}

 Now applying Theorem~\ref{fi} to $E \to Y'$, we have the following description of the Ohsawa measure in terms of the discriminant divisor $B_R$ supported on $B$ as a consequence of Theorem~\ref{fi}.

\begin{theorem}\label{dimage}
  Let $(X, \Psi)$ and $Y$ be as above, with a unique lc place as in \eqref{EYX}.   The Ohsawa measure $dV[\Psi]_Y$ is less singular (as continuous singular volume forms on $\yr$) than the direct image $\pi_* v$ under $\pi: Y' \to Y$ of a singular volume form $v$ on $Y'$ with poles along Kawamata's discriminant divisor $B_R$ on $Y'$ multiplied by a local psh weight with vanishing Lelong numbers.
  
\end{theorem}

   Equivalently, $v$ can be locally written as  (up to a bounded positive factor) 
	  
	  \begin{equation*}\label{asy1}
	   v(w) =  ( \prod^m_{i=1}  \abs{w_i}^{-2 a_i} )  e^{-\psi(w)} \abs{dw_1 \wedge \ldots \wedge dw_m}^2   
	  \end{equation*}
 in local analytic coordinates $w = (w_1, \ldots, w_m)$ adapted to $B$ on an open subset  $U \subset Y'$ (with $\dim Y' = \dim Y = m$)	  where  $\psi$ is a psh function on $U$ with vanishing Lelong numbers and $\sum a_i \divisor (w_i) = (B_R)|_U$.

 \begin{proof}
 
 As noted above,  by  \cite[Prop. 3.2]{K21}, the Ohsawa measure $dV[\Psi]_Y$ is equal to the direct image $(\mu|_E)_* d\nu$ where $d\nu$ is a measure corresponding to a singular volume form with poles along the snc divisor $R := F|_E$. 
 
 Both $d\nu$ and $dV[\Psi]_Y$ have the property of putting no mass on closed analytic subsets, \cite[Prop. 3.2, Prop. 3.5]{K21}. Hence these measures are determined by their restrictions to nonempty Zariski open subsets. Therefore it suffices to check the property of `less singular' on the restriction of $f: E \to Y'$ to $f_0 : f^{-1} (Y_0) \to Y_0$ where $Y_0 \subset Y'$ is the subset of regular values of $f$. 
 
 Now it is a matter of applying Theorem~\ref{fi} (for $f: X \to Y$ in the notation therein) to $f: E \to Y'$ at hand. The conclusion follows from $f_* d\nu$ (`the Ohsawa measure on $Y'$') being less singular than $v$, which results from  $d\nu$ (resp. $\beta$) having the divisorial pole along $R = F|_E$ (resp. along $R + \lceil -R_h \rceil$) while $v$ at hand corresponds to $\theta = f_* \beta$ of Theorem~\ref{fi}. 
 \end{proof} 
  
\begin{remark1}
Remark~\ref{mild} applies also here. Note that even though $Y$ has a unique lc place,  there may exist a component of $R_h$ which comes from a component of $F$ whose image under $\mu$ is equal to $Y$ but with discrepancy $> -1$. 

\end{remark1}

   \subsection{Comparison of two $L^2$ extension theorems} \hfill \\

 Using Theorem~\ref{dimage}, we establish a precise comparison between two  $L^2$ extension theorems of \cite[Thm. 4.2]{K07} and \cite[Thm. 2.8]{D15} for log canonical pairs (for a smooth complex projective variety $X$). 
For  this comparison, we first derive \cite[Thm. 3.9]{K21}, a version of $L^2$ extension for a maximal log canonical center, from \cite[Thm. 2.8]{D15}. For the convenience of readers, we recall the statement.

\begin{theorem}\label{extension}  \cite[Thm. 2.8 and Remark 2.9 (b)]{D15}, cf.  \cite[Thm. 3.9]{K21}. 

Let $(X, \omega)$ be a weakly pseudoconvex K\"ahler manifold. 	
	 Let an lc pair $(X, \psi)$, a $\QQ$-line bundle $L$, a psh metric $\psi$ for $L$ and a quasi-psh function $\Psi$ be  as in Definition~\ref{Omeasure}, so that $ e^{-\psi} = h e^{-\Psi}$. 
	Assume that, for some $\delta > 0$,

\begin{equation}\label{curvature}
   i \Theta(L, h) +  \alpha \iddb \Psi =   i \Theta (L,\psi) +  (\alpha -1) \iddb \Psi  \ge 0
\end{equation}

\noi for all $\alpha \in [1, 1+\delta]$. 
 Let $B$ be a $\QQ$-line bundle on $X$ such that $K_X + L + B$ is a $\ZZ$-line bundle. Let $b$ be a psh metric of $B$.

Let $Y$ be a maximal lc center of $(X, \psi)$ with a unique lc place. If
  a holomorphic section $s \in H^0 (Y, (K_X + L + B)|_Y)$ satisfies

\begin{equation}\label{input1}
 \int_{\yr} \abs{s}_{\omega, h, b}^2  dV[\Psi]_Y < \infty ,
\end{equation}

\noi  then there exists a holomorphic section $\wt{s} \in H^0 ( X, K_X + L + B) $ such that we have $\tilde{s}|_Y = s$ and moreover

\begin{equation}\label{output1}
 \int_X   \abs{\wt{s}}_{\omega, h, b}^2  \gamma (\delta \Psi) e^{-\Psi}  dV_{\omega}        \le   \frac{34}{\delta}  \int_{\yr} \abs{s}_{\omega, h, b}^2  dV[\Psi]_Y.  
 \end{equation}
 
\noindent  Here the Ohsawa measure $dV[\Psi]_Y$ is taken with respect to the smooth volume form $dV_\omega$ and $\gamma$ is as in \cite[(2.7)]{D15}.

\end{theorem}

Now we assume that $X$ is projective. 
 Let $\norm{s}_1 $ be the input norm of \cite[Thm. 4.2]{K07}   which was given as the  $L^2$ norm for an adjoint line bundle with respect to a Kawamata metric~\cite[Def. 3.1]{K07}.
  Let $\norm{s}_2$ be the input norm in \cite[Thm. 3.9]{K21} with respect to the Ohsawa measure $dV[\Psi]_Y$. From Theorem~\ref{dimage}, we have the following comparison of two $L^2$ extension theorems from \cite{K07} and \cite{D15}.

 \begin{theorem} \label{kawaresi}	  
	 Let $Y$ be a maximal lc center of $(X, \psi)$   with a unique lc place. 
 Let $B$ be a $\QQ$-line bundle on $X$ such that $K_X + L + B$ is a $\ZZ$-line bundle. Let $b$ be a psh metric of $B$.    Let $s$ be a holomorphic section of $K_Y + M + B|_Y = (K_X + L + B)|_Y$ on $Y$ (where $K_Y + M$ is defined as in \cite[Def. 3.1]{K07}).

  Let $p \in Y$ be a point. 
 If the  $L^2$  norm of $ s$  with respect to a Kawamata metric $e^{-\kappa}$ is locally finite at $p$, i.e.    there exists a neighborhood  $V \subset Y$ of $p$ such that   
 
 $$ \norm{s}_1 :=  \int_V \abs{s}^2 \cdot e^{-\kappa} \cdot b|_Y < \infty ,$$

\noi then the $L^2$ norm of $s$  with respect to the Ohsawa measure is locally finite at $p$, i.e.  there exists a neighborhood  $U \subset Y$ of $p$ such that  
 $$ \norm{s}_2 :=  \int_U  \abs{s}_{\omega, h, b}^2  dV[\Psi]_Y < \infty. $$ 

\end{theorem}

\begin{proof}

We will use the notation as in the setting of  \eqref{EYX}.  In particular, we have the relation $K_{Y'} +  M' = \pi^* (K_Y + M)$ where   $\pi : Y' \to Y$ is the modification as before. 
 
 From the definition of the Kawamata metric $e^{-\kappa}$ \cite[Def. 3.1]{K07}, locally we can write $e^{-\kappa} = e^{-\vp} e^{-\eta}$ as a metric for $M'$ on $Y'$ where $e^{-\vp}$ is a singular hermitian metric given by the discriminant divisor and $e^{-\eta}$ is a choice of a smooth metric. On the other hand, by Theorem~\ref{dimage}, the Ohsawa measure  has density  locally  equal to the product $e^{-\vp} e^{-\chi}$ where $\chi$ is a local psh weight with vanishing Lelong numbers. Let $\beta$ be a psh function such that $b = e^{-\beta}$ near $p$. 
 Then by Lemma~\ref{monat} (2), we have the equivalence:
 
  $$\int_{\pi^{-1}(V)} \abs{\pi^* s}^2 e^{-\vp} e^{-\eta} e^{-\beta}  < \infty \text{\;\; if and only if \;\;} \int_{\pi^{-1}(U)} \abs{\pi^* s}^2 e^{-\vp} e^{-\chi}  e^{-\beta} < \infty $$ where $U$ and $V$ are as in the statement. This completes the proof.  Note that $\vp$ is locally the difference of two psh functions with analytic singularities, which is why we need (2) of Lemma~\ref{monat}. 
 \end{proof}

\begin{lemma}  \label{monat}
 Let $u$ and $v$ be two psh functions on a complex manifold $X$. Suppose that $v$ has zero Lelong numbers at every point of $X$. Then the following hold:  
 
 (1) For the multiplier ideals, we have $\JJ(u+v) = \JJ(u)$. 
 
 (2) More generally, let $\psi$ be a psh function with analytic singularities. We have $e^{\psi  - u} \in L^2_{\loc} $ if and only if $e^{\psi - u - v} \in L^2_{\loc}$. 

\end{lemma}

\begin{proof} 

It is clear that (2) implies (1).   (1) is from \cite[Prop. 2.3]{K15} and (2) from \cite[Cor. 10.15]{B20}. Both of them use the strong openness theorem of \cite{GZ15}. 
\end{proof}

\begin{remark1}

Theorem~\ref{kawaresi} means that the  geometric understanding of the Ohsawa measure obtained in Theorem~\ref{dimage1} enables  Demailly's $L^2$ extension theorem (\cite[Thm. 2.8]{D15}, \cite[Thm. 3.9]{K21}) to serve as generalization of the $L^2$ extension in  \cite[Thm. 4.2]{K07}   in that
  a strictly positive curvature condition (``small ample'' $A$) in \cite[Thm. 4.2]{K07} is replaced with a semipositive one in \cite{D15} (see \eqref{curvature} in Theorem~\ref{extension}).    
\end{remark1}

\quad\\

\section{Applications to semipositivity theorems} 

In this section, we will give an  application of the main results to semipositivity theorems,  Theorem~\ref{Kawa00}. Before that, we first revisit the elliptic fibrations.

\subsection{Elliptic fibrations} \label{elliptic0}  \hfill \\

\noi  The classical elliptic fibrations studied by Kodaira~\cite{K63}, \cite{K64} provide the important initial case of the canonical bundle formula. Even in this case, our main result seems new, cf. Theorem~\ref{elliptic}.

Let $f: X \to Y$ be a relatively minimal elliptic fibration with $\dim Y = 1$ (see e.g. \cite[Thm. V.12.1]{BPV}). When $f$ has multiple fibers $m_1 F_1, \ldots, m_k F_k$, we have the canonical bundle formula 

\begin{equation}\label{elliptic1}
K_X = f^* ( K_Y + G) + \sum (m_i -1) F_i 
\end{equation}

\noi where $G$ is the line bundle equal to $ f_* (K_{X / Y})$.  We turn this into the following equality of $\QQ$-line bundles: 
  
\begin{equation}\label{elliptic2}  
   K_X = f^* \left( K_Y +  G +  \sum^k_{i=1} \frac{m_i -1}{m_i} Q_i \right) 
\end{equation}

\noi where $Q_i \in Y$ is a point viewed as a divisor such that $f^* Q_i = m_i F_i$ as Cartier divisors.   This is a situation where we can apply Theorem~\ref{pt} viewing \eqref{elliptic2} as $K_X + L = f^* (K_Y + M)$ with $L$ trivial and $M := G + \OO(\sum^k_{i=1} \frac{m_i -1}{m_i} Q_i)$.  Note that in general, the equality $K_{X/Y} = f^*M$ holds as $\QQ$-line bundles only : the direct image $f_* (K_{X/Y})$ is a $\ZZ$-line bundle which is not necessarily equal to $M$.   

 Furthermore, by \cite[(2.9)]{F86}, \cite[(8.2.1)]{Ko07} (also see e.g. the introduction to \cite{FM00}), we have equality of $\QQ$-line bundles 
 
 \begin{equation}\label{nonmultiple}
 G = \frac{1}{12} j^* \OO_{\PP^1} (1) + \OO( \sum_{k \in K} \sigma_k P_k )
 \end{equation}
 where $\sigma_k$ is the well-known coefficients (see e.g. \cite[(2.6)]{F86}) from the list of singular fibers \cite{K63} and $j : Y \to \PP^1$ is the map into the moduli.

The trivial metric $1$ is a psh metric for the trivial line bundle $L$, thus we get the corresponding $L^2$ metric $\mu$ for the $\QQ$-line bundle $M$ in a canonical way.

 \begin{theorem}\label{elliptic}
 
  In this case of an elliptic fibration $f: X \to Y$, the $L^2$ metric $\mu$ is the product of a singular hermitian metric given by the  divisor $\sum_{k \in K} \sigma_k P_k  + \sum^k_{i=1} \frac{m_i -1}{m_i} Q_i $ and a singular hermitian metric with vanishing Lelong numbers for $\frac{1}{12} j^* \OO_{\PP^1} (1)$.

 \end{theorem}

\begin{proof}

 Let $g: X' \to X$ be a proper birational morphism so that Kawamata's condition (\ref{snc}) is satisfied for $X' \to Y$ : we will apply Theorem~\ref{main}. The $L^2$ metric for $X' \to Y$ coincides with the $L^2$ metric $\mu$. 

 Define the divisor $R$ (in Theorem~\ref{main}) by  $K_{X'}  + R = g^* K_X$. For dimension reasons, $R_h = 0$.  The discriminant divisor $B_R$ is equal to $\sum_{k \in K} \sigma_k P_k  + \sum^k_{i=1} \frac{m_i -1}{m_i} Q_i $. Thus as $\QQ$-line bundles, $J$ (in Theorem~\ref{main}) is equal to $\frac{1}{12} j^* \OO_{\PP^1} (1) $. This concludes the proof. 
\end{proof}

\begin{remark1}

 We note that already in this case, it is  natural to adopt the generality of $\QQ$-line bundles to be able to equally deal with multiple and non-multiple singular fibers in the canonical bundle formula, as in our results Theorem~\ref{main}, Theorem~\ref{pt}.  One can compare, for example, with  \cite[Prop. 2.1]{EFM} where  the $L^2$ metric is for the line bundle $f_* (K_{X / Y})$ which seems to count non-multiple singular fibers only (as in \eqref{nonmultiple}). 
 \end{remark1}

\subsection{An alternative proof of a semipositivity theorem} \hfill \\

\noi We  recall the following series of semipositivity theorems  for an algebraic fiber space $f: X \to Y$ (i.e. a surjective morphism of smooth projective varieties with connected fibers) under some general conditions. 

\begin{enumerate}
\item
 \cite[Thm. 5]{Ka81} : Nefness of the locally free sheaf $E := f_* K_{X/Y}$.   (Cf. \cite{F78}, \cite{FFS}.)

  \item
 \cite[Thm. 2]{Ka98} : Log version of (1) for log Calabi-Yau fibrations $f$.

 \item
 
 \cite[Thm. 1.1]{Ka00}, \cite{FF17} : Refinement of (1) replacing nefness by the existence of a singular hermitian metric with vanishing Lelong numbers for $\OO_{\PP(E)} (1)$.  
 
\end{enumerate}

 We showed in Theorem~\ref{main}  that the moduli part line bundle in the canonical bundle formula \cite[Thm. 2]{Ka98} for $f: X \to Y$ admits a psh metric with vanishing Lelong numbers when $R_h \ge 0$. In some cases, the moduli part can coincide with the direct image $f_* (K_{X / Y} ) $ of the relative canonical line bundle of $f$. 
 
 Using this, we obtain as a corollary of Theorem~\ref{main}, an alternative proof of the following semipositivity theorem of Kawamata~\cite[Thm. 1.1 (3)]{Ka00}.  This alternative proof does not use some highly nontrivial results such as the SL2-orbit theorem from the theory of variation of Hodge structure, cf. \cite{CKS86}. 
 
 \begin{remark1}  The only place where Hodge theory was used in the proof of our main results is limited to Lemma~\ref{BR} which has a short proof in \cite[8.4.9 (3)]{Ko07}, in particular, not depending on \cite{CKS86}. 
 \end{remark1}

\begin{theorem}\cite[Thm. 1.1 (3)]{Ka00}, cf. \cite[Cor. 1.2]{FF17} \label{Kawa00}
 Let $f:X \to Y$ be a surjective  morphism with connected fibers between smoth complex projective varieties. Let $B$ be an snc divisor on $Y$ such that $f$ restricted over $Y \setminus B$ is a holomorphic submersion. Let $X_0 := f^{-1} (Y \setminus B)$. Let $n := \dim X - \dim Y$. 
 \begin{itemize}
 \item
 Suppose that a general smooth fiber $F$ satisfies $K_F \sim 0$. 
 \item
    Suppose that $R^n f_* \CC_{X_0}$ has unipotent monodromies around the components of $B$.
  \end{itemize}
    
\noi     Then the line bundle $f_* (K_{X \slash Y})$ admits a singular hermitian metric with vanishing Lelong numbers.   In particular,  $f_* (K_{X \slash Y})$ is nef. 
 
\end{theorem}

The fact that $f_* (K_{X \slash Y})$ is a line bundle in this setting is due to \cite[\S 4]{Ka81}. We follow the exposition of \cite[(8.4.4)]{Ko07}.   

\begin{proof}

 First observe that, by \cite[Lem. 8.3.4]{Ko07}, there exists a vertical divisor $R$ on $X$ such that $K_X + R$ is $\QQ$-linearly equivalent to the pullback by $f$ of a $\QQ$-line bundle on $Y$. 
 Let $\mu: X' \to X$ be a proper modification (given by composition of blow-ups) such that $f': X' \to X \to Y$ satisfies Kawamata's condition (\ref{snc}).  Then one can write (for a divisor $R'$ on $X'$) that $K_{X'} + R' = \mu^* (K_X + R)   = f'^* (K_Y + J(X' / Y, R') + B_{R'} )$ from \cite[Thm. 8.3.7]{Ko07}. 
 
 As \cite[Thm. 8.3.7 (1)]{Ko07} states, the moduli part depends only on the generic fibers $F$ of $f$ and the pairs on them $(F, R_h |_F)$. In the present case, $R$ and $R'$ are vertical divisors, thus we have $R_h = 0, R'_h= 0$. For a generic fiber $F'$ of $X' \to Y$, we have isomorphism $F' \to F$ and in view of the pairs $(F, R_h |_F) = (F, 0),  (F', R'_h |_F') = (F', 0)$,  we have 
 
 \begin{equation}\label{JJ}
 J(X' \slash Y, R' ) = J(X \slash Y, R)
 \end{equation}
  on $Y$ where both sides are defined in \cite[(8.4.6)]{Ko07}.

  Now considering the open set $X'_0 \subset X'$ that is isomorphic to $X_0$ under the proper modification $X' \to X$ (indeed $X' \to X$ can be taken to satisfy this), we have $R^n f_* \CC_{X_0} = R^n f'_* \CC_{X'_0}$. Since the unipotent monodromies condition is satisfied for both of them, from \cite[(8.4.6)]{Ko07}, we have $J(X' \slash Y, R') = f'_* (K_{X' \slash Y})$ and $J(X/Y, R) = f_* (K_{X \slash Y})$ respectively.  
  
   Combining with \eqref{JJ}, we have $f_* (K_{X \slash Y}) \cong J(X' \slash Y, R')$. Since $J(X' \slash Y, R')$ admits a psh metric with vanishing Lelong numbers by Theorem~\ref{main},  so does $f_* (K_{X \slash Y})$. 
\end{proof}

\qa
\\

\normalsize

\noi \textsc{Dano Kim}

\noi Department of Mathematical Sciences and Research Institute of Mathematics

\noi Seoul National University, 08826  Seoul, Korea

\noi Email address: kimdano@snu.ac.kr

\end{document}